\documentclass{amsart}
\numberwithin{equation}{section}

\usepackage{amssymb}
\usepackage{enumerate, xspace}
\usepackage{marvosym} %monetary symbols, \EUR or \EURcr, \EyesDollar
\usepackage[sans]{dsfont}        %allows \mathds{E 1}, without sans is less heavy
\usepackage{changebar}
\usepackage[backgroundcolor=blue!20!white, linecolor=blue!20!white,
textsize=footnotesize]{todonotes}
\usepackage[colorlinks]{hyperref}
\newtheorem{thm}{Theorem}[section]
\newtheorem{lem}[thm]{Lemma}

\newtheorem{prop}[thm]{Proposition}

\newtheorem{defin}{Definition}

\newtheorem{rem}[thm]{Remark}

\newcommand\cA{{\mathcal A}}
\newcommand\cB{{\mathcal B}}
\newcommand\cC{{\mathcal C}}
\newcommand\cD{{\mathcal D}}

\newcommand\cF{{\mathcal F}}

\newcommand\cK{{\mathcal K}}
\newcommand\cL{{\mathcal L}}
\newcommand\cO{{\mathcal O}}

\newcommand\cS{{\mathcal S}}

\newcommand\cW{{\mathcal W}}

\newcommand\bE{{\mathbb E}}
\newcommand\bG{{\mathbb G}}
\newcommand\bN{{\mathbb N}}
\newcommand\bP{{\mathbb P}}

\newcommand\bR{{\mathbb R}}
\newcommand\bT{{\mathbb T}}

\newcommand\frkL{{\mathfrak L}}

\newcommand\ve{\varepsilon}

\newcommand\vf{\varphi}

\newcommand\Const{C_\#}
\newcommand\const{c_\#}
\newcommand\Id{{\mathds{1}}}

\newcommand{\st}{\;|\;}
\newcommand\nc[1]{_{\cC^{#1}}}
\newcommand{\deh}{{\textup{d}}}
\newcommand{\stdf}{\frkL}

\newcommand{\invr}{^{-1}}

\newcommand\tell{{\tilde\ell}}
\newcommand{\vei}{\ve^{-1}}
\newcommand{\nt}{\text{c}}
\newcommand{\expb}{b}
\newcommand{\expo}[1]{\exp(#1)}
\newcommand{\expa}{a}

\newcommand{\Dtr}{\mathfrak D}%{\cD}
\newcommand{\ThetaAvg}{{\overline\Theta}}

\newcommand{\sigmaold}{s}
\newcommand{\nuold}{v}
\newcommand{\muold}{r}

\begin{document}
\title[Martingales]{The martingale approach\\ after\\
  Varadhan and Dolgopyat }
\author{Jacopo De Simoi}
\address{Jacopo De Simoi\\
  Department of Mathematics\\
  University of Toronto\\
  40 St George St. Toronto, ON M5S 2E4}
\email{{\tt jacopods@math.utoronto.ca}}
\urladdr{\href{http://www.math.utoronto.ca/jacopods}{http://www.math.utoronto.ca/jacopods}}
\author{Carlangelo Liverani}
\address{Carlangelo Liverani\\
  Dipartimento di Matematica\\
  II Universit\`{a} di Roma (Tor Vergata)\\
  Via della Ricerca Scientifica, 00133 Roma, Italy.}  \email{{\tt liverani@mat.uniroma2.it}}

\date\today
%\date{ {\bf File: {\jobname}.tex.}}% eliminate in final version
\begin{abstract}
  We present, in the simplest possible form, the so called {\em martingale problem} strategy to establish limit
  theorems. The presentation is specially adapted to problems arising in partially hyperbolic dynamical systems. We will
  discuss a simple partially hyperbolic example with fast-slow variables and use the martingale method to prove an
  averaging theorem and study fluctuations from the average. The emphasis is on ideas rather than on results. Also, no
  effort whatsoever is done to review the vast literature of the field.
\end{abstract}
\thanks{ C.L. thanks Konstantin Khanin for pointing out (long time ago, during the 2011 {\em Thematic Program on
    Dynamics and Transport in Disordered Systems} at the Fields Institute) the need for this note and forcing him to
  start thinking about it. We thank Mikko Stenlund and Denis Volk for several comments on a preliminary version.  We
  thank Dmitry Dolgopyat, Alexey Korepanov, Zemer Kosloff, Ian Melbourne, Mark Pollicott and the anonymous referees for
  providing several suggestions which improved the readability of the current version.  We thank the Centre Bernoulli,
  Lausanne, Switzerland where part of this notes were written. Both authors have been partially supported by the ERC
  Grant MALADY (ERC AdG 246953). J.D.S. acknowledges partial NSERC support}
\maketitle

%%%PAPER
\tableofcontents
\section{Introduction}
In this note\footnote{ A first, preliminary, version of this note was prepared by the second author for a
  mini course at the conference {\em Beyond Uniform Hyperbolicity} in Bedlewo, Poland, held at the end of May 2013,
  which, ultimately, he could not attend.  The note was then extended and presented during the semester {\em Hyperbolic
    dynamics, large deviations and fluctuations} held at the Bernoulli Centre, Lausanne, January--June 2013.} we purport
to explain in the simplest possible terms a strategy to investigate the statistical properties of dynamical systems put
forward by Dmitry Dolgopyat~\cite{D1}.  It should be remarked that Dolgopyat has adapted to the field of Dynamical
Systems a scheme developed by Srinivasa Varadhan and collaborators first for the study of stochastic process arising
from a diffusion~\cite{SV}, then for the study of limit theorems (e.g. the hydrodynamics limit), starting with the
pioneering \cite{GPV}, and large deviations, e.g. \cite{DV}.\footnote{ It should be noted that the above has no
  pretension of being an exact historical reconstruction, it just describes the way we learned this material. Indeed,
  some of the relevant ideas were previously present. See, e.g., the reference~\cite{Kha} pointed out to us by Sergei
  Kuksin.}  The adaptation is highly non trivial as in the case of Dynamical Systems two basic tools commonly used in
probability (conditioning and It\=o calculus) are missing.  The lesson of Dolgopyat is that such tools can be recovered
nevertheless, provided one looks at the problem in the {\em right} way.

Rather than making an abstract exposition, we prefer a hands-on presentation.  Hence, we will illustrate the method by
discussing a super simple (but highly non trivial) example.

The presentation is especially aimed at readers in the field of Dynamical Systems. Thus probabilists could find the exposition at times excessively detailed and/or redundant and at other times a bit too fast.

\subsection{Fast-Slow partially hyperbolic systems}\

\noindent We are interested in studying fast-slow systems in which the fast variable undergoes a strongly chaotic
motion.  Namely, let $M, S$ be two compact Riemannian manifolds, let $X=M\times S$ be the configuration space of our
systems and let $m_{\textrm{Leb}}$ be the Riemannian measure on $M$.  For simplicity, we consider only the case in which $S=\bT^d$
for some $d\in\bN$.  We consider a map $F_0\in\cC^r(X, X)$, $r\geq 3$, defined by
\[
F_0(x,\theta)=(f(x, \theta),\theta)
\]
where the maps $f(\cdot,\theta)$ are uniformly hyperbolic for every $\theta$.  If we consider a small perturbation of
$F_0$ we note that the perturbation of $f$ still yields a uniformly hyperbolic system, by structural stability. Thus
such a perturbation can be subsumed in the original maps.  Hence, it suffices to study families of maps of the form
\[
F_\ve(x,\theta)=(f(x,\theta), \theta+\ve\omega(x,\theta))
\]
with $\ve\in (0,\ve_0)$, for some $\ve_0$ small enough, and $\omega\in\cC^r$.

Such systems are called fast-slow, since the variable $\theta$, the slow variable, needs a time at least $\ve^{-1}$ to change substantially.

The basic question is {\bf what are the statistical properties of $F_\ve$?}

The answer to such a question is at the end of a long road that starts with the attempt to understand the dynamics
for times of order $\ve^{-1}$.  In this note we will concentrate on such a preliminary problem and will describe how to
overcome the first obstacles along the path we would like to walk.

\subsection{The unperturbed system: $\ve=0$}\label{sec:unperturbed}\

\noindent The statistical properties of the system are well understood in the case $\ve=0$. In such a case $\theta$ is
an invariant of motion, while for every $\theta$ the map $f(\cdot, \theta)$ has strong statistical properties. We
will need such properties in the following discussion which will be predicated on the idea that, for times long but much
shorter than $\vei$, on the one hand $\theta$ remains almost constant, while, on the other hand, its change depends
essentially on the behavior of an ergodic sum with respect to a fixed dynamics $f(\cdot, \theta)$.  It is not obvious which exact general properties are necessary to prove the type of results we are interested in. Yet, let us give an idea of the situation by listing the main properties that we will need, and use, in the following.

\begin{enumerate}
\item the maps $f(\cdot, \theta)$ admit a unique SRB (Sinai--Ruelle--Bowen) measure $m_\theta$.
\item the measure $m_\theta$, when seen as an element of $\cC^1(M,\bR)'$, is differentiable in $\theta$.
\item there exists $C_0,\alpha>0$ such that, for each $g,h\in \cC^1(M,\bR)$, we have\footnote{ We remark that a slower
    decay of correlation could suffice, but let us keep things simple.}
\[
\begin{split}
  &\left| m_{\textrm{Leb}}(h\cdot g\circ f^n(\cdot,\theta))-m_\theta(g)m_{\textrm{Leb}}(h)\right|\leq C_0 e^{-\alpha n}\|h\|_{\cB_1}\|g\|_{\cB_2},\\
  &\left| m_{\theta}(h\cdot g\circ f^n(\cdot,\theta))-m_\theta(g)m_{\theta}(h)\right|\leq C_0 e^{-\alpha
    n}\|h\|_{\cB_1}\|g\|_{\cB_2},
\end{split}
\]
where $\cB_1,\cB_2$ are appropriate Banach spaces.\footnote{ The exact required properties for the Banach spaces vary
  depending on the context. In the context that we are going to consider nothing much is needed. Yet, in general, it could be helpful to have
  properties that allow to treat automatically multiple correlations: let $\{g_1,g_2,g_3\}\subset C^1$, then
\[
m_{\textrm{Leb}}(g_1\cdot (g_2\circ f^n\cdot g_3)\circ f^m)= m_{\textrm{Leb}}(g_1)m_\theta(g_2\circ f^n\cdot
g_3)+\cO(e^{-\alpha m}\|g_1\|_{\cB_1}\|g_2\circ f^n \cdot g_3\|_{\cB_2}).
\]
Thus, in order to have automatically decay of multiple correlations we need, at least, $ \|g_2\circ f^n \|_{\cB_2}\leq
\Const \|g_2\|_{\cB_2}$, which is false, for example, for the $\cC^1$ norm.}
\end{enumerate}

The above properties hold for a wide class of uniformly hyperbolic systems, \cite{babook, GL1, GL2, BT}, yet here, to further
simplify the exposition, we assume that $M=\bT^1$ and
\begin{equation}\label{eq:hyperbolicity}
\partial_x f\geq \lambda>2.
\end{equation}
Then a SRB measure is just a
measure absolutely continuous with respect to Lebesgue and all the above properties are well known with the choices
$\cB_1=\cC^1$ and $\cB_2=\cC^0$ or $\cB_1=\operatorname{BV}$ and $\cB_2=L^1$ (see \cite{L1} for a fast and elementary
exposition or \cite{babook} for a more complete discussion).
\begin{rem} For the wondering reader: in all the following arguments the case of an higher dimensional expanding map can be treated in almost exactly the same way (a part from a slightly heavier notation).\footnote{ Simply, the support of a standard pair will be a ball rather than a segment.} On the contrary, the case of a hyperbolic map is a bit more complex (although the logic of the argument remains exactly the same) due to the different definition of standard pairs necessary to handle the stable direction. See \cite{D1} for details.
\end{rem}
\begin{rem} Note that in the following we do not require or use the exact knowledge of the spectrum of the transfer
  operator.\footnote{ The transfer operator $\cL_\theta$ is simply the adjoint of the dynamics,
    i.e. $\cL_\theta\mu(g)=\mu(g\circ f(\cdot,\theta))$, when acting on an appropriate class of measures.} Yet, a detailed
  understanding of the transfer operator might be necessary in order to obtain sharper results.
\end{rem}

It follows that the dynamical systems $(X, F_0)$ has uncountable many SRB measures: all the measures of the form
$\mu(\vf)=\int \vf(x,\theta) m_\theta(dx)\nu(d\theta)$ for an arbitrary measure $\nu$. The ergodic measures are the ones
in which $\nu$ is a point mass. The system is partially hyperbolic and has a central foliation. Indeed, the $f(\cdot,
\theta)$ are all topologically conjugate by structural stability of expanding maps \cite{KH}. Let $h(\cdot,\theta)$ be
the map conjugating $f(\cdot,0)$ with $f(\cdot,\theta)$, that is $h(f(x,0),\theta)=f(h(x,\theta),\theta)$. Thus the
foliation $W^c_x=\{ (h(x,\theta), \theta)\}_{\theta\in S}$ is invariant under $F_0$ and consists of points that stay,
more or less, always at the same distance, hence it is a center foliation.  Note however that, since in general $h$ is
only a H\"older continuous function (see~\cite{KH}) the foliation is very irregular and, typically, not absolutely
continuous.

In conclusion, the map $F_0$ has rather poor statistical properties and a not very intuitive description as a partially
hyperbolic system. It is then not surprising that its perturbations form a very rich universe to explore and already the
study of the behavior of the dynamics for times of order $\vei$ (a time long enough so that the variable $\theta$ has a
non trivial evolution, but far too short to investigate the statistical properties of $F_\ve$) is interesting and non
trivial.

\section{Preliminaries and results}\label{sec:prelim}
Let $\mu_0$ be a probability measure on $X$.  Let us define $(x_n,\theta_n)=F_\ve^n(x,\theta)$, then $(x_n,\theta_n)$
are random variables~\footnote{ Recall that a {\em random variable} is a measurable function from a probability space to
  a measurable space.} if $(x_0,\theta_0)$ are distributed according to $\mu_0$.\footnote{ That is, the probability
  space is $X$ equipped with the Borel $\sigma$-algebra, $\mu_0$ is the probability measure and $(x_n,\theta_n)$ are
  functions of $(x,\theta)\in X$.}  It is natural to define the polygonalization\footnote{ Since we interpolate between
  close points the procedure is uniquely defined in $\bT$.}
\begin{equation}\label{eq:path}
\Theta_\ve(t)=\theta_{\lfloor\ve^{-1} t\rfloor}+(t- \ve\lfloor\ve^{-1} t\rfloor)(\theta_{\lfloor\ve^{-1} t\rfloor+1}-\theta_{\lfloor\ve^{-1} t\rfloor}),\quad t\in[0,T].
\end{equation}
Note that $\Theta_\ve$ is a random variable on $X$ with values in $\cC^0([0,T],S)$. Also, note the time rescaling done so that one expects non trivial paths.

It is often convenient to consider random variables defined directly on the space $\cC^0([0,T],S)$ rather than $X$. Let us
discuss the set up from such a point of view.  The space $\cC^0([0,T],S)$ endowed with the uniform topology is a
separable metric space. We can then view $\cC^0([0,T],S)$ as a probability space equipped with the Borel $\sigma$-algebra. It
turns out that such a $\sigma$-algebra is the minimal $\sigma$-algebra containing the open sets
$\bigcap_{i=1}^n\{\vartheta\in \cC^0([0,T],S)\st \vartheta(t_i)\in U_i\}$ for each $\{t_i\}\subset [0,T]$ and open sets $U_i\subset S$, \cite[Section
1.3]{SV}. Since $\Theta_\ve$ can be viewed as a continuous map from $X$ to $\cC^0([0,T], S)$, the measure $\mu_0$
induces naturally a measure $\bP^\ve$ on $\cC^0([0,T],S)$: $\bP^\ve=(\Theta_\ve)_*\mu_0$.\footnote{ Given a measurable
  map $T:X\to Y$ between measurable spaces and a measure $P$ on $X$, $T_*P$ is a measure on $Y$ defined by
  $T_*P(A)=P(T^{-1}(A))$ for each measurable set $A\subset Y$.}  Also, for each $t\in [0,T]$ let $\Theta(t)\in
\cC^0(\cC^0([0,T],S),S)$ be the random variable defined by $\Theta(t,\vartheta)=\vartheta(t)$, for each $\vartheta\in\cC^0([0,T],S)$. Next, for
each $\cA\in\cC^0(\cC^0([0,T],S),\bR)$, we will write $\bE^\ve(\cA)$ for the expectation with respect to $\bP^\ve$.
For $A\in \cC^0(S,\bR)$ and $t\in[0,T]$, $\bE^\ve(A\circ \Theta(t))=\bE^\ve(A( \Theta(t)))$ is the expectation of the function
$\cA(\vartheta)=A(\vartheta(t))$, $\vartheta\in\cC^0([0,T],S)$.

To continue, a more detailed discussion concerning the initial conditions is called for.  Note that not all measures are
reasonable as initial conditions.  Just think of the possibility to start with initial conditions given by a point mass,
hence killing any trace of randomness. The best one can reasonably do is to fix the slow variable and leave the
randomness only in the fast one. Thus we will consider measures $\mu_0$ of the following type: for each
$\vf\in\cC^0(X,\bR)$, $\mu_0(\vf)=\int \vf(x,\theta_0) h(x)dx$ for some $\theta_0\in S$ and $h\in\cC^1(M,\bR_+)$. Our
first problem is to understand $\lim_{\ve\to 0}\bP^\ve$. After some necessary preliminaries, in Section
\ref{sec:average} we will prove the following result..
\begin{thm}\label{thm:average} The measures $\{\bP^\ve\}$ have a weak limit $\bP$, moreover $\bP$ is a measure supported
  on the trajectory determined by the O.D.E.
\begin{equation}\label{eq:average0}
\begin{split}
  \dot{\ThetaAvg}=\bar\omega(\ThetaAvg)\\
  \ThetaAvg(0)=\theta_0
\end{split}
\end{equation}
where $\bar\omega(\theta)=\int_M \omega(x,\theta) m_\theta(dx)$.
\end{thm}
The above theorem specifies in which sense the random variable $\Theta_\ve$ converges to the average dynamics described
by equation \eqref{eq:average0}.

The next natural question is how fast the convergence takes place.  To this end it is natural to consider, for each $t\in[0,T]$,
\[
\zeta_\ve(t)=\ve^{-\frac 12}\left[\Theta_\ve(t)-\ThetaAvg(t)\right].
\]
Note that $\zeta_\ve$ is a random variable on $X$ with values in $\cC^0([0,T],\bR^d)$ which describes the fluctuations around the
average.\footnote{ Here we are using that $S=\bT^d$ can be lifted to its universal cover $\bR^d$.} Let
$\widetilde\bP^\ve$ be the path measure describing $\zeta_\ve$ when $(x_0,\theta_0)$ are distributed according to the
measure $\mu_0$.  That is, $\widetilde\bP^\ve=(\zeta_\ve)_*\mu_0$.  Our second task, and the last in this note, will be
to understand the limit behavior of $\widetilde\bP^\ve$, hence of the fluctuation around the average. Section \ref{sec:fluctuations} will be devoted to proving the following result.
\begin{thm}\label{thm:diffusion} The measures $\{\widetilde\bP^\ve\}$ have a weak limit $\widetilde\bP$. Moreover,
  $\widetilde\bP$ is the measure of the zero average Gaussian process defined by the Stochastic Differential Equation (SDE)
\begin{equation}\label{eq:diffusion0}
\begin{split}
&d\zeta=D\bar\omega(\ThetaAvg)\zeta dt +\sigma(\ThetaAvg)dB\\
&\zeta(0)=0,
\end{split}
\end{equation}
where $B$ is the $\bR^d$ dimensional standard Brownian motion and the diffusion coefficient $\sigma$ is given
by~\footnote{ In our notation, for any measure $\mu$ and vectors $v,w$, $\mu(v\otimes w)$ is a matrix with entries
  $\mu(v_i w_j)$.}
\begin{equation}\label{e_definitionBarChi}
\begin{split}
  \sigma(\theta)^2 =& m_\theta\left(\hat\omega(\cdot,\theta)\otimes\hat\omega(\cdot,\theta)\right)+ \sum_{m=1}^{\infty}
   m_\theta\left( \hat\omega(f_\theta^m(\cdot),\theta)\otimes\hat\omega(\cdot,\theta)\right)+\\
   &+ \sum_{m=1}^{\infty}m_\theta\left( \hat\omega(\cdot,\theta)\otimes\hat\omega(f_\theta^m(\cdot),\theta)\right).
\end{split}
\end{equation}
where $\hat\omega=\omega-\bar\omega$ and we have used the notation $f_\theta(x)=f(x,\theta)$.  In addition, $\sigma^2$
is symmetric and non-negative, hence $\sigma$ is uniquely defined as a symmetric positive definite matrix. Finally,
$\sigma(\theta)$ is strictly positive, unless $\hat\omega(\theta, \cdot)$ is a coboundary for $f_\theta$.
\end{thm}

\begin{rem}\label{rem:gauss-sol} Note that, setting $\psi(\lambda,t)=\bE(e^{i\langle\lambda,\zeta(t)\rangle})$, equation \eqref{eq:diffusion0} implies, by It\=o's formula, that
\[
\begin{split}
&\partial_t\psi=\langle\lambda, D\bar\omega \partial_\lambda\psi\rangle-\frac 12\langle \lambda,\sigma^2\lambda\rangle \psi\\
&\psi(\lambda, 0)=1
\end{split}
\]
which implies that $\psi$ is a zero mean Gaussian. In turn, this implies that $\zeta$ is a zero mean Gaussian process, see the proof of Proposition \ref{lem:unique} for more details.
\end{rem}

\begin{rem}
  It is interesting to notice that equation \eqref{eq:diffusion0} with $\sigma\equiv 0$ is just the equation for the
  evolution of an infinitesimal displacement of the initial condition, that is the linearised equation along an orbit of
  the averaged deterministic system.  This is rather natural, since in the time scale we are considering, the fluctuations around
  the deterministic trajectory are very small.
\end{rem}

\begin{rem}
  Note that the condition that insures that the diffusion coefficient $\sigma$ is non zero can be constructively checked
  by finding periodic orbits with different averages.
\end{rem}

Having stated our goals, let us begin with a first, very simple, result.
\begin{lem}\label{lem:tight} The measures $\{\bP^\ve\}$ are tight.
\end{lem}
\begin{proof}
By~\eqref{eq:path} it follows that the path $\Theta_\ve$ is made of segments of length $\ve$ and maximal slope $\|\omega\|_{L^\infty}$, thus for all $h>0$,\footnote{ The reader should be aware that we use the notation $\Const$ to designate a
    generic constant (depending only on $f$ and $\omega$) which numerical value can change from one occurrence to the
    next, even in the same line.}
\[
\|\Theta_\ve(t+h)-\Theta_\ve(t)\|\leq \Const h+\ve\sum_{k=\lceil\ve^{-1}t\rceil}^{\lfloor\ve^{-1}(t+h)\rfloor-1}\|\omega(x_k,\theta_k)\|\leq \Const h.
\]
Thus the measures $\bP^\ve$ are all supported on a set of uniformly Lipschitz functions, that is a compact set.
\end{proof}
The above means that there exist converging subsequences $\{\bP^{\ve_j}\}$.
Our next step is to identify the set of accumulation points.

An obstacle that we face immediately is the impossibility of using some typical probabilistic tools. In particular, conditioning with respect to the past and It\=o's formula. In fact, even if the initial condition is random, the dynamics is still deterministic, hence conditioning with respect to the past seems hopeless as it might kill all the randomness at later times.

To solve the first problem it is therefore necessary to devise a systematic way to use the strong dependence on the initial condition (typical of hyperbolic systems) to show that the dynamics, in some sense, {\em forgets the past}. One way of doing this effectively is to use standard pairs, introduced in the next section, whereby slightly enlarging our allowed initial conditions. Exactly how this solves the conditioning problem will be explained in Section \ref{sec:condition}. The lack of It\=o's formula will be overcome by taking the point of view of the Martingale problem to define the solution of a SDE. To explain what this means in the present context is the goal of the present note, but see Appendix \ref{sec:moi} for a brief comment on this issue in the simple case of an SDE. We will come back to the problem of studying the accumulation points of $\{\bP^\ve\}$ after having settled the issue of conditioning.

%%%%%%%%%%%%%%%%%%%%%%%%%%%%%%%%%
\section{Standard Pairs}\label{sec:standard}

Let us fix $\delta>0$ small enough, and $\Dtr>0$ large enough, to be specified later; for $c_1>0$ consider the set of
functions
\begin{align*}
  \Sigma_{c_1}=\{G\in \cC^2([a,b], S) \st &a,b\in\bT^1, b-a\in[\delta/2, \delta],\\&\|G'\|\nc0\leq \ve c_1,\,
  \|G''\|\nc0\leq \ve \Dtr c_1,\}.
\end{align*}
Let us associate to each $G\in \Sigma_{c_1}$ the map $\bG\in \cC^2([a,b], X)$ defined by $\bG(x)=(x,G(x))$ whose image is a
curve --the graph of $G$-- which will be called  a \emph{standard curve}.  For $c_2>0$ large enough, let us define the set of
$c_2$-\emph{standard} probability densities on the standard curve as
\[
D_{c_2}(G)=\left\{\rho\in \cC^1([a,b],\bR_+)\;\bigg|\; \int_a^b\rho(x)d x=1,\
  \left\|\frac{\rho'}{\rho}\right\|_{\cC^0}\leq c_2\right\}.
\]
A \emph{standard pair} $\ell$ is given by $\ell=(\bG,\rho)$ where $G\in\Sigma_{c_1}$ and $\rho\in D_{c_2}(G)$.  Let
$\overline \frkL_\ve$ be the collection of all standard pairs for a given $\ve>0$.  A standard pair
$\ell=(\bG,\rho)$ induces a probability measure $\mu_\ell$ on $X=\bT^{d+1}$ defined as follows: for any continuous function
$g$ on $X$ let
\[
\mu_\ell(g):=\int_a^{b} g(x,G(x))\rho(x) d x.
\]
We define\footnote{ This is not the most general definition of standard family, yet it suffices for our purposes.} a
\emph{standard family} $\stdf=(\cA,\nu, \{\ell_j\}_{j\in\cA})$, where $\cA\subset\bN$ and $\nu$ is a probability measure
on $\cA$; i.e.\ we associate to each standard pair $\ell_j$ a positive weight $\nu(\{j\})$ so that
$\sum_{j\in\cA}\nu(\{j\})=1$.  For the following we will use also the notation $\nu_{\ell_j}=\nu(\{j\})$ for each
$j\in\cA$ and we will write $\ell\in\stdf$ if $\ell=\ell_j$ for some $j\in\cA$. A standard family $\stdf$ naturally
induces a probability measure $\mu_\stdf$ on $X$ defined as follows: for any measurable function $g$ on $X$ let
\[
\mu_\stdf(g):=\sum_{\ell\in\stdf}\nu_{\ell}\mu_\ell(g).
\]
Let us denote by $\sim$ the equivalence relation induced by the above correspondence i.e.\ we let $\stdf\sim\stdf'$ if
and only if $\mu_\stdf=\mu_{\stdf'}$.

\begin{prop}[Invariance]\label{p_invarianceStandardPairs}
  There exist $\delta$ and $\Dtr$ such that, for any $c_1$, $c_2$ sufficiently large, and $\ve$ sufficiently small,
  for any standard family $\stdf$, the measure $F_{\ve*}\mu_{\stdf}$ can be decomposed in standard pairs, i.e.\ there
  exists a standard family $\stdf'$ such that $F_{\ve*}\mu_{\stdf}=\mu_{\stdf'}$.  We say that $\stdf'$ is a
  \emph{standard decomposition of}
  $F_{\ve*}\mu_{\stdf}$.
  \end{prop}
\begin{proof}
  For simplicity, let us assume that $\stdf$ is given by a single standard pair $\ell$; the general case does not require
  any additional ideas and it is left to the reader.  By definition, for any measurable function $g$:
  \begin{align*}
    F_{\ve*}\mu_\ell(g) &= \mu_{\ell}( g\circ F_{\ve})=\\&=
    \int_{a}^{b}%
    g(f(x,G(x)),G(x)+\ve\omega(x,G(x)))\cdot  %
    \rho(x)%
    d x .
  \end{align*}
  It is then natural to introduce the map $f_{\bG}:[a,b]\to \bT^1$ defined by $f_{\bG}(x)=f\circ\bG(x)$.  Note that, by assumption \eqref{eq:hyperbolicity},
  $f_{\bG}'\geq \lambda-\ve c_1\|\partial_\theta f\|\nc0>3/2$ provided that $\ve$ is small enough (depending on how
  large is $c_1$).  Hence all $f_{\bG}$'s are expanding maps, moreover they are invertible if $\delta$ has been chosen
  small enough.  In addition, for any sufficiently smooth function $A$ on $X$, it is trivial to check that, by the
  definition of standard curve, if $\ve$ is small enough (once again depending on $c_1$)\footnote{ Given a function $A$ by $\deh A$ we mean the differential.}
  \begin{subequations}\label{e_AG}
    \begin{align}
      \|(A\circ\bG)'\|\nc0 &\leq \|\deh A\|\nc0 + \ve\|\deh A\|\nc0 c_1\label{e_AG'}\\
      \|(A\circ\bG)''\|\nc0 &\leq 2\|\deh A\|\nc1 + \ve\|\deh A\|\nc0\Dtr c_1 \label{e_AG''}.
    \end{align}
  \end{subequations}
  Then, fix a partition (mod $0$) $[f_{\bG}(a),f_{\bG}(b)]=\bigcup_{j=1}^m[a_{j},b_{j}]$,
  with $b_{j}-a_{j}\in[\delta/2,\delta]$ and $b_{j}=a_{j+1}$; moreover let
  $\vf_{j}(x)=f_{\bG}\invr(x)$ for $x\in[a_{j},b_{j}]$ and define
  \begin{align*}
    G_{j}(x) &= G\circ\vf_{j}(x)+\ve\omega(\vf_{j}(x),G\circ\vf_{j}(x));\\
    \tilde\rho_{j}(x) &= \rho\circ\vf_{j}(x)\vf_{j}'(x).
  \end{align*}
  By a change of variables we can thus write:
  \begin{equation}\label{e_induction0}
    F_{\ve*}\mu_\ell(g) =\sum_{j=1}^m\int_{a_{j}}^{b_{j}} \tilde\rho_j(x) g(x,G_j(x))d x.
  \end{equation}
  Observe that, by immediate differentiation we obtain, for $\vf_j$:
  \begin{align}\label{e_estimatesVf}
    \vf'_j   &= \frac1{f_\bG'}\circ\vf_j &%
    \vf''_j  &= -\frac{f_\bG''}{f_\bG'^3}\circ\vf_j .%
  \end{align}
  Let $\omega_{\bG}=\omega\circ \bG$ and $\bar G=G+\ve\omega_{\bG}$. Differentiating the definitions of $G_j$ and
  $\tilde\rho_j$ and using~\eqref{e_estimatesVf} yields
  \begin{align}
    G_j' &=%
    \frac{\bar G'}{f'_\bG}\circ\vf_j\label{e_C1}&%
    G_j'' &=%
    \frac{\bar G''}{f'^2_\bG}\circ\vf_j%
    -G_j'\cdot\frac{f_\bG''}{f_\bG'^2}\circ\vf_j%\label{e_c1''}%
  \end{align}
  and similarly
  \begin{align}\label{e_C2}
    \frac{\tilde\rho_j'}{\tilde\rho_j}&=\frac{\rho'}{\rho\cdot f'_\bG}\circ\vf_j-\frac{f_\bG''}{f_\bG'^2}\circ\vf_j.
  \end{align}
  Using the above equations it is possible to conclude our proof: first of all, using~\eqref{e_C1}, the definition of
  $\bar G$ and equations~\eqref{e_AG} we obtain, for small enough $\ve$:
  \begin{align*}
    \|G'_j\| &\leq \left\|\frac{ G' + \ve\omega_\bG'}{f_\bG'}\right\|
    \leq\frac23(1+\Const\ve)\ve c_1 + \Const\ve \leq \\%
    &\leq\frac34\ve c_1 + \Const\ve\leq\ve c_1,\\
    \intertext{provided that $c_1$ is large enough; then:}
    \|G''_j\| &\leq \left\|\frac{ G'' + \ve\omega_\bG''}{f_\bG'^2}\right\|+\Const(1+\ve\Dtr c_1)\ve c_1 \leq \\
    &\leq \frac34\ve \Dtr c_1 + \ve \Const  c_1 + \ve \Const  \leq \ve \Dtr c_1
  \end{align*}
provided $c_1$ and $\Dtr$ are sufficiently large. Likewise, using~\eqref{e_AG}
  together with~\eqref{e_C2} we obtain
  \begin{align*}
    \left\|\frac{\tilde\rho'_j}{\tilde\rho_j}\right\|&\leq \frac23 c_2 +  \Const(1+\Dtr c_1) \leq  c_2,
  \end{align*}
  provided that $c_2$ is large enough.  This concludes our proof: it suffices to define the family $\stdf'$ given by
  $(\cA,\nu, \{\ell_j\}_{j\in\cA})$, where $A=\{1,\dots,m\}$, $\nu(\{j\})=\int_{a_j}^{b_j}\tilde\rho_j$,
  $\rho_j=\nu(\{j\})^{-1}\tilde\rho_j$ and $\ell_j=(\bG_j,\rho_j)$.  Our previous estimates imply that
  $(\bG_{j},\rho_{j})$ are standard pairs; note moreover that~\eqref{e_induction0} implies
  $\sum_{\tell\in\stdf'}\nu_\tell=1$, thus $\stdf'$ is a standard family.  Then we can rewrite~\eqref{e_induction0} as
  follows:
  \[
  F_{\ve *}\mu_\ell(g)=%
  \sum_{\tell\in\stdf'}\nu_\tell\mu_{\tell}(g)=%
  \mu_{\stdf'}(g).\qedhere
  \]
\end{proof}
\begin{rem}\label{r_randomVariables}
  Given a standard pair $\ell=(\bG,\rho)$, we will interpret $(x_k,\theta_k)$ as random
  variables defined as $(x_k,\theta_k)=F_\ve^k(x, G(x))$, where $x$ is distributed
  according to $\rho$.
\end{rem}

%%%%%%%%%%%%%%%%%%%%%%%%%%

\section{Conditioning}\label{sec:condition}
In probability, conditioning is one of the most basic techniques and one would like to use it freely when dealing with
random variables.  Yet, as already mentioned, conditioning seems unnatural when dealing with deterministic systems.  The
use of standard pairs provides a very efficient solution to this conundrum. The basic idea is that one can apply
repeatedly Proposition \ref{p_invarianceStandardPairs} to obtain at each time a family of standard pairs and then
``condition" by specifying to which standard pair the random variable belongs at a given time.\footnote{ Note that the
  set of standard pairs does not form a $\sigma$-algebra, so to turn the above into a precise statement would be a bit
  cumbersome. We thus prefer to follow a slightly different strategy, although the substance is unchanged.}

Note that if $\ell$ is a standard pair with $G'=0$, then it belongs to $\overline \frkL_\ve$ for all $\ve>0$.  In the
following, abusing notations, we will use $\ell$ also to designate a family $\{\ell_\ve\}$, $\ell_\ve\in \overline
\frkL_\ve$ that weakly converges to a standard pair $\ell\in\bigcap_{\ve>0}\overline \frkL_\ve$.  For every standard
pair $\ell$ we let $\bP_\ell^\ve$ be the induced measure in path space and $\bE_\ell^\ve$ the associated expectation.

Before continuing, let us recall and state a bit of notation: for each $t\in [0,T]$ recall that the random variable $\Theta(t)\in
\cC^0(\cC^0([0,T],S),S)$ is defined by $\Theta(t,\vartheta)=\vartheta(t)$, for all $\vartheta\in \cC^0([0,T],S)$. Also we will need the
filtration of $\sigma$-algebras $\cF_t$ defined as the smallest $\sigma$-algebra for which all the functions
$\{\Theta(s)\;:\; s\leq t\}$ are measurable.  Last, we consider the shift $\tau_s:\cC^0([0,T],S)\to \cC^0([0,T-s],S)$
defined by $\tau_s(\vartheta)(t)=\vartheta(t+s)$.
Note that $\Theta(t)\circ \tau_s=\Theta(t+s)$. Also, it is helpful to keep in mind that, for all $A\in\cC^0(S,\bR)$, we have\footnote{ To be really precise, maybe one should write, e.g.,  $\bE_\ell^\ve(A\circ \Theta(t+k\ve))$, but we conform to the above more intuitive notation.}
\[
\bE_\ell^\ve(A(\Theta(t+k\ve)))=\mu_\ell(A(\Theta_\ve(t+k\ve)))=\mu_\ell(A(\Theta_\ve(t)\circ F_\ve^{k})).
\]
Our goal is to compute, in some reasonable way, expectations of $\Theta(t+s)$ conditioned to $\cF_t$, notwithstanding
the above mentioned problems due to the fact that the dynamics is deterministic. Obviously, we can hope to obtain a
result only in the limit $\ve\to 0$. Note that we can always reduce to the case in which the conditional expectation is
zero by subtracting an appropriate function, thus it suffices to analyze such a case.

The basic fact that we will use is the following.

\begin{lem}\label{lem:condition} Let $t'\in [0,T]$ and $\cA$ be a continuous bounded random variable on $\cC^0([0,t'],S)$ with values in $\bR$.
  If we have
\[
\lim_{\ve\to 0}\sup_{\ell\in\overline \frkL_\ve}\left| \bE_\ell^\ve(\cA)\right|=0,
\]
then, for each $s\in [0,T-t']$, standard pair $\ell$, uniformly bounded continuous functions $\{B_{i}\}_{i=1}^m$,
$B_i:S\to\bR$ and times $\{t_1<\cdots<t_m\}\subset [0,s)$,
\[
\lim_{\ve\to 0}\bE_\ell^\ve\left(\prod_{i=1}^m B_{i}(\Theta(t_i))\cdot \cA\circ \tau_s\right)=0.
\]
\end{lem}
\begin{proof}
  The quantity we want to study can be written as
\[
\mu_\ell\left(\prod_{i=1}^m B_{i}(\Theta_\ve(t_i))\cdot \cA(\tau_s(\Theta_\ve))\right).
\]
To simplify our notation, let $k_i=\lfloor t_i\ve^{-1}\rfloor$ and $k_{m+1}=\lfloor s\ve^{-1}\rfloor$.  Also, for every
standard pair $\tilde \ell $, let $\frkL_{i,\tilde \ell}$ denote an arbitrary standard decomposition of
$(F_\ve^{k_{i+1}-k_i})_*\mu_{\tilde\ell}$ and define $\theta_\ell^*=\mu_\ell(\theta)=\int_{a_\ell}^{b_\ell}\rho_\ell(x)
G_\ell(x)dx$.  Then, by Proposition \ref{p_invarianceStandardPairs},
\[
\begin{split}
&\mu_\ell\left(\prod_{i=1}^m B_{i}(\Theta_\ve(t_i))\cdot \cA(\tau_s(\Theta_\ve))\right)
=\mu_\ell\left(\prod_{i=1}^m B_{i}(\Theta_\ve(t_i))\cdot \cA(\tau_{s-\ve k_{m+1}}(\Theta_\ve\circ F_\ve^{k_{m+1}}))\right)\\
&=\sum_{\ell_1\in\frkL_{1,\ell}}\cdots \sum_{\ell_{m+1}\in\frkL_{m,\ell_{m}}}\left[\prod_{i=1}^m\nu_{\ell_i} B_{i}(\theta_{\ell_i}^*)\right]\nu_{m+1}\mu_{\ell_{m+1}}(\cA(\Theta_\ve))+o(1)\\
&=\sum_{\ell_1\in\frkL_{1,\ell}}\cdots \sum_{\ell_{m+1}\in\frkL_{m,\ell_{m}}}\left[\prod_{i=1}^m\nu_{\ell_i} B_{i}(\theta_{\ell_i}^*)\right]\nu_{m+1}\bE^\ve_{\ell_{m+1}}(\cA)+o(1)
\end{split}
\]
where $\lim_{\ve\to 0}o(1)=0$.  The lemma readily follows.
\end{proof}
Lemma \ref{lem:condition} implies that, calling $\bP$ an accumulation point of $\bP_\ell^\ve$, we have\footnote{ By $\bE$ we mean the expectation with respect to $\bP$.}
\begin{equation}\label{eq:cond-2}
\bE \left(\prod_{i=1}^m B_i(\Theta(t_i))\cdot \cA\circ\tau_s\right)=0.
\end{equation}
This solves the conditioning problems thanks to the following
\begin{lem}\label{lem:limit-cond}
Property \eqref{eq:cond-2} is equivalent to
\[
\bE \left( \cA\circ \tau_s\;|\;\cF_s\right)=0,
\]
for all $s<t$.
\end{lem}
\begin{proof}
Note that the statement of the Lemma immediately implies \eqref{eq:cond-2}, we thus worry only about the other direction.
If the lemma were not true then there would exist a positive measure set of the form
\[
\cK=\bigcap_{i=0}^\infty \{\vartheta(t_i)\in K_i\},
\]
where the $\{K_i\}$ is a collection of compact sets in $S$, and $t_i<s$, on which the conditional expectation is strictly positive (or strictly negative, which can be treated in exactly the same way).  For some arbitrary $\delta>0$, consider open sets $U_i\supset K_i$ be such that $\bP(\{\vartheta(t_i)\in U_i\setminus K_i\})\leq \delta 2^{-i}$.  Also, let $B_{\delta,i}$ be a continuous function such that $B_{\delta,i}(\vartheta)=1$ for $\vartheta\in K_i$ and $B_{\delta,i}(\vartheta)=0$ for $\vartheta\not\in U_i$.  Then
\[
0<\bE(\Id_{\cK} \cA\circ\tau_s)=\lim_{n\to \infty}\bE\left(\prod_{i=1}^n B_{\delta,i}(\Theta(t_i))\cdot \cA\circ\tau_s\right)+C_\#\delta=C_\#\delta
\]
which yields a contradiction by the arbitrariness of $\delta$.
\end{proof}
In other words, we have recovered the possibility of conditioning with respect to the past {\em after} the limit $\ve\to 0$.

\section{Averaging (the Law of Large Numbers)}\label{sec:average}
We are now ready to provide the proof of Theorem~\ref{thm:average}. The proof consists of several steps; we first
illustrate the global strategy while momentarily postponing the proof of the single steps.

\begin{proof}[{\bf Proof of Theorem \ref{thm:average}}]
  As already mentioned we will prove the theorem for a larger class of initial conditions: any initial condition
  determined by a standard pair. Note that for flat standard pairs $\ell$, i.e. $G_\ell(x)=\theta$, we have the class of
  initial condition assumed in the statement of the Theorem. Given a standard pair $\ell$ let $\{\bP_\ell^\ve\}$ be the
  associate measures in path space (the latter measures being determined, as explained at the beginning of Section
  \ref{sec:prelim}, by the standard pair $\ell$ and~\eqref{eq:path}).  We have already seen in Lemma
  \ref{lem:tight} that the set $\{\bP_\ell^\ve\}$ is tight.

  Next we will prove in Lemma \ref{lem:average} that, for each $A\in\cC^2(S,\bR)$, we have
\begin{equation}\label{eq:average-mart}
\lim_{\ve\to 0}\sup_{\ell\in\overline\frkL_\ve}\left|\bE_\ell^\ve\left(A(\Theta(t))-A(\Theta(0))-\int_0^t\langle\overline\omega(\Theta(\tau)),\nabla A(\Theta(\tau))\rangle d\tau\right)\right|=0.
\end{equation}
Accordingly, it is natural to consider the random variables $\cA(t)$ defined by
\[
\cA(t,\vartheta)=A(\vartheta(t))-A(\vartheta(0))-\int_0^t\langle\overline\omega(\vartheta(\tau)),\nabla A(\vartheta(\tau))\rangle d\tau,
\]
for each $t\in[0,T]$ and $\vartheta\in\cC^0([0,T],S)$, and the first order differential operator
\[
\cL A=\langle\overline\omega,\nabla A\rangle.
\]
Then equation \eqref{eq:average-mart}, together with Lemmata \ref{lem:condition} and \ref{lem:limit-cond}, means that each accumulation point $\bP_\ell$ of $\{\bP_\ell^\ve\}$ satisfies, for all $s\in[0,T]$ and $t\in [0,T-s]$,
\begin{equation}\label{eq:martingale}
\bE_\ell\left(\cA\circ \tau_s\;|\;\cF_s\right)=
\bE_\ell\left(A(\Theta(t+s))-A(\Theta(s))-\int_s^{t+s}\hskip-6pt\cL A(\Theta(\tau)) d\tau\;\bigg|\;\cF_s\right)=0
\end{equation}
this is the simplest possible version of the {\em Martingale Problem}.  Indeed it implies that, for all $\theta, A$ and standard pair $\ell$ such that $G_\ell(x)=\theta$,
\[
M(t)=A(\Theta(t))-A(\Theta(0))-\int_0^t\cL A(\Theta(s)) ds
\]
is a martingale with respect to the measure $\bP_\theta$ and the filtration $\cF_t$ (i.e., for each $0\leq s\leq t\leq T$, $\bE_\theta(M(t)\st\cF_s)=M(s)$).\footnote{ We use $\bP_\theta$ to designate any measure $\bP_\ell$ with $G_\ell(x)=\theta$.}  Finally we will show in
Lemma~\ref{lem:martingale} that there is a unique measure that has such a property: the measure supported on the unique
solution of equation \eqref{eq:average0}. This concludes the proof of the theorem.
\end{proof}
In the rest of this section we provide the missing proofs.

\subsection{Differentiating with respect to time}\

\noindent Let us start with the proof of \eqref{eq:average-mart}.
\begin{lem}\label{lem:average} For each $A\in\cC^2(S,\bR)$ we have
\[
\lim_{\ve\to 0}\sup_{\ell\in\overline\frkL_\ve}\left|\bE_\ell^\ve\left(A(\Theta(t))-A(\Theta(0))-\int_0^t\langle\overline\omega(\Theta(s)),\nabla A(\Theta(s))\rangle ds\right)\right|=0,
\]
where (we recall) $\overline\omega(\theta)=m_{\theta}(\omega(\cdot,\theta))$ and $m_\theta$ is the unique SRB measure of $f(\cdot,\theta)$.
\end{lem}
\begin{proof}
We will use the notation of Appendix \ref{sec:shadow}. Given a standard pair $\ell$ let $\rho_\ell=\rho$, $\theta^*_\ell=\mu_\ell(\theta)$ and $f_*(x)=f(x,\theta^*_\ell)$. Then, by Lemmata \ref{lem:shadow-q} and \ref{lem:der-bounds-geo}, we can write, for $n\leq C\ve^{-\frac 12}$,\footnote{ By $\cO(\ve^a n^b)$ we mean a quantity bounded by $\Const \ve^a n^b$, where $\Const$ does not depend on $\ell$.}
\[
\begin{split}
&\mu_\ell\left(A(\theta_n)\right)=\int_a^b\rho(x) A\left(\theta_0+\ve\sum_{k=0}^{n-1}\omega(x_k,\theta_k)\right) dx\\
&=\int_{a}^b\rho(x) A\left(\theta^*_\ell+\ve\sum_{k=0}^{n-1}\omega(x_k,\theta^*_\ell)\right) dx+\cO(\ve^2n^2+\ve)\\
&=\int_a^b\rho(x) A(\theta^*_\ell)dx+\ve\sum_{k=0}^{n-1}\int_a^b\rho(x) \langle\nabla A(\theta^*_\ell),\omega(x_k,\theta^*_\ell)\rangle dx+\cO(\ve)\\
&=\int_a^b\rho(x) A(G_\ell(x))dx+\cO(\ve)\\
&\quad+\ve\sum_{k=0}^{n-1}\int_a^b\rho(x) \langle\nabla A(\theta^*_\ell),\omega(f_*^k\circ Y_n(x),\theta^*_\ell)\rangle dx\\
&=\mu_\ell(A(\theta_0))+\ve\sum_{k=0}^{n-1}\int_{\bT^1}\tilde\rho_n(x) \langle\nabla A(\theta^*_\ell),\omega(f_*^k(x),\theta^*_\ell)\rangle dx+\cO(\ve)
\end{split}
\]
where $\tilde\rho_n(x)=\left[\frac{\chi_{[a,b]}\rho}{Y_n'}\right]\circ Y_n^{-1}(x)$.  Note that $\int_{\bT^1}\tilde
\rho_n=1$ but, unfortunately, $\|\tilde\rho\|_{BV}$ may be enormous.  Thus, we cannot estimate the integral in the above
expression by naively using decay of correlations.  Yet, equation \eqref{eq:deriv-Y} implies $|Y'_n-1|\leq C_\# \ve
n^2$.  Moreover, $\bar\rho= (\chi_{[a,b]}\rho)\circ Y^{-1}$ has uniformly bounded variation.\footnote{ Indeed, for all
  $\vf\in\cC^1$, $|\vf|_\infty\leq 1$, $\int\bar\rho\vf'=\int_a^b\rho \cdot \vf'\circ Y\cdot Y'=\int_a^b\rho (\vf\circ
  Y)'\leq \|\rho\|_{BV}$.}  Accordingly, by the decay of correlations and the $\cC^1$ dependence of the invariant
measure on $\theta$ (see Section~\ref{sec:unperturbed}) we have
\[
\begin{split}
\int_{\bT^1}\tilde\rho_n(x) \langle\nabla A(\theta^*_\ell), \omega&(f_*^k(x),\theta^*_\ell)\rangle dx=\int_{\bT^1}\bar\rho_n(x) \langle\nabla A(\theta^*_\ell),\omega(f_*^k(x),\theta^*_\ell)\rangle dx+\cO(\ve n^2)\\
&=m_{\textrm{Leb}}(\tilde\rho_n(x))m_{\theta^*_\ell}\left(\langle\nabla A(\theta^*_\ell),\omega(\cdot,\theta^*_\ell)\rangle\right)+\cO(\ve n^2+e^{-\const k})\\
&=\mu_\ell\left(\langle\nabla A(\theta_0),\overline\omega(\theta_0)\rangle\right)+\cO(\ve n^2+e^{-\const k}).
\end{split}
\]
Accordingly,
\begin{equation}\label{eq:one-step-mu}
\mu_\ell\left(A(\theta_n)\right)=\mu_\ell(A(\theta_0)+\ve n\langle \nabla A(\theta_0),\bar\omega(\theta_0)\rangle)+\cO(n^3\ve^2+\ve).
\end{equation}
Finally, we choose $n=\lceil\ve^{-\frac 13}\rceil$ and set $h=\ve n$.  We define inductively standard families such that
$\frkL_{\ell_0}=\{\ell\}$ and for each standard pair $\ell_{i+1}\in \frkL_{\ell_i}$ the family $\frkL_{\ell_{i+1}}$ is a
standard decomposition of the measure $(F_\ve^{n})^*\mu_{\ell_{i+1}}$.  Thus, setting $m=\lceil t\ve^{-\frac 23}\rceil
-1$, recalling equation~\eqref{eq:one-step-mu} and using repeatedly Proposition~\ref{p_invarianceStandardPairs},
 \[
\begin{split}
  &\bE_\ell^\ve(A(\Theta(t)))=\mu_\ell(A(\theta_{t\ve^{-1}}))=\mu_\ell(A(\theta_0))+\sum_{k=0}^{m-1}\mu_\ell(A(\theta_{\ve^{-1}(k+1)h}))-A(\theta_{\ve^{-1}kh}))\\
  &=\mu_\ell(A(\theta_0))+\sum_{k=0}^{m-1}\sum_{\ell_1\in\frkL_{\ell_0}}\dots\sum_{\ell_{k-1}\in\frkL_{\ell_{k-2}}}
  \prod_{j=1}^{k-1}\nu_{\ell_j}
  \left[\mu_{\ell_{k-1}}(\ve^{\frac 23}\langle \nabla A(\theta_{0}),\bar\omega(\theta_{0})\rangle)+\cO(\ve )\right]\\
  &=\bE_\ell^\ve\left(A(\Theta(0))+\sum_{k=0}^{m-1}\langle \nabla A(\Theta(kh)),\bar\omega(\Theta(kh))\rangle h\right)+\cO(\ve^{\frac 13} t)\\
  &=\bE_\ell^\ve\left(A(\Theta(0))+\int_0^t\langle \nabla A(\Theta(s)),\bar\omega(\Theta(s))\rangle
    ds\right)+\cO(\ve^{\frac 13} t ).
\end{split}
\]
The lemma follows by taking the limit $\ve\to 0$.
\end{proof}

\subsection{The Martingale Problem at work}\

\noindent First of all let us specify precisely what we mean by the {\em martingale problem}.
\begin{defin}[Martingale Problem]\label{def:martingale} Given a Riemannian manifold $S$, a linear operator $\cL:\cD(\cL)\subset \cC^0( S,\bR^d)\to \cC^0(S,\bR^d)$, a set of measures $\bP_y$, $y\in S$, on $\cC^0([0,T], S)$ and a filtration $\cF_t$ we say that $\{\bP_y\}$ satisfies the  martingale problem if for each function $A\in \cD(\cL)$,
\[
\begin{split}
  &\bP_y(\{z(0)=y\})=1\\
  &M(t,z):=A(z(t))-A(z(0))-\int_0^t\cL A(z(s))ds \text{ is } \cF_t \text{-martingale under all } \bP_y .
\end{split}
\]
\end{defin}
We can now prove the last announced result.
\begin{lem}\label{lem:martingale}
If $\bar\omega$ is Lipschitz, then
the martingale problem determined by \eqref{eq:martingale} has a unique solution consisting of the measures supported on the solutions of the ODE
\begin{equation}\label{eq:average}
\begin{split}
&\dot{\ThetaAvg}=\overline\omega(\ThetaAvg)\\
&\ThetaAvg(0)=y.
\end{split}
\end{equation}
\end{lem}
\begin{proof}
Let $\ThetaAvg$ be the solution of \eqref{eq:average} with initial condition $y\in\bT^d$ and $\bP_y$ the probability measure in the martingale problem.  The idea is  to compute
\[
\begin{split}
\frac d{dt}\bE_y(\|\Theta(t)-\ThetaAvg(t)\|^2)=&\frac d{dt}\bE_y(\langle \Theta(t),\Theta(t)\rangle)-2\langle \bar\omega(\ThetaAvg(t)),\bE_y(\Theta(t))\rangle\\
&-2\langle \ThetaAvg(t),\frac d{dt}\bE_y( \Theta(t))\rangle+ 2\langle \bar\omega(\ThetaAvg(t)),\ThetaAvg(t))\rangle.
\end{split}
\]
To continue we use Lemma \ref{lem:basic-m} where, in the first term $A(\theta)=\|\theta\|^2$, in the third $A(\theta)=\theta_i$ and the generator in \eqref{eq:martingale} is given by $\cL A(\theta)=\langle \nabla A(\theta),\bar\omega(\theta)\rangle$.
\[
\begin{split}
\frac d{dt}\bE_y(\|\Theta(t)-\ThetaAvg(t)\|^2)&=2\bE_y(\langle \Theta(t),\bar\omega(\Theta(t))\rangle)-2\langle \bar\omega(\ThetaAvg(t)),\bE_y(\Theta(t))\rangle\\
&\quad-2\langle \ThetaAvg(t),\bE_y(\bar\omega(\Theta(t)))\rangle+2\bE_y(\langle \ThetaAvg(t),\bar\omega(\ThetaAvg(t))\rangle)\\
&=\bE_y(\langle\Theta(t)-\ThetaAvg(t),\overline\omega(\Theta(t))-\overline\omega(\ThetaAvg(t))\rangle).
\end{split}
\]
By the Lipschitz property of $\bar\omega$ (let $C_L$ be the Lipschitz constant), using the Schwartz inequality and
integrating we have
\[
\bE_y(\|\Theta(t)-\ThetaAvg(t)\|^2)\leq 2C_L\int_0^t \bE_y(\|\Theta(s)-\ThetaAvg(s)\|^2)ds
\]
which, by Gronwall's inequality, implies that
\[
\bP_y(\{\ThetaAvg\})=1.\qedhere
\]
\end{proof}

\section{A recap of what we have done so far}\label{sec:recap}
We have just seen that the martingale method (in Dolgopyat's version) consists of four steps
\begin{enumerate}
\item Identify a suitable class of measures on path space which allow one to handle the conditioning problem (in our case: the one coming from standard pairs)
\item Prove tightness for such measures (in our case: they are supported on uniformly Lipschitz functions)
\item Identify an equation characterizing the accumulation points (in our case: an ODE)
\item Prove uniqueness of the limit equation in the martingale sense.
\end{enumerate}
The beauty of the previous scheme is that it can be easily adapted to a variety of problems.  To convince the reader of this fact we proceed further and apply it to obtain more refined information on the behavior of the system.

\section{Fluctuations (the Central Limit Theorem)}\label{sec:fluctuations}
It is possible to study the limit behavior of $\zeta_\ve$ using the strategy summarized in Section \ref{sec:recap},
even though now the story becomes technically more involved.  Let us discuss the situation a bit more in detail.  Let
$\widetilde\bP^\ve_\ell$ be the path measure describing $\zeta_\ve$ when $(x_0,\theta_0)$ are distributed according to
the standard pair $\ell$.\footnote{ As already explained, here we allow $\ell$ to stand also for a family $\{\ell_\ve\}$
  which weakly converges to $\ell$. In particular, this means that $\ThetaAvg$ is also a random variable, as it
  depends on the initial condition $\theta_0$.}  Note that, $\widetilde\bP^\ve_\ell=(\zeta_\ve)_*\mu_\ell$.  Again, we
provide a proof of the claimed results based on some facts that will be proven in later sections.

\begin{proof}[{\bf Proof of Theorem \ref{thm:diffusion}}]
  First of all, the sequence of measures $\widetilde\bP^\ve_\ell$ is tight, which will be proven in Proposition
  \ref{lem:tight1}.  Next, by Proposition \ref{prop:diff-eq}, we have that
\begin{equation}\label{eq:uniform-martingale}
\lim_{\ve\to 0}\sup_{\ell\in\overline\frkL_\ve}\left|\widetilde\bE_\ell^\ve\left(A(\zeta(t))-A(\zeta(0))-\int_0^t\cL_s A(\zeta(s)) ds\right)\right|=0,
\end{equation}
where
\begin{equation}\label{eq:secondorderop}
(\cL_s A)(\zeta)=\langle \nabla A(\zeta), D\bar\omega(\ThetaAvg(s)) \zeta\rangle +\frac 12\sum_{i,j=1}^d[\sigma^2(\ThetaAvg(s))]_{i,j} \partial_{\zeta_i}\partial_{\zeta_j}A(\zeta),
\end{equation}
with diffusion coefficient $\sigma^2$ given by \eqref{e_definitionBarChi}. In the following we will often write, slightly abusing notations,
$\sigma(t) $ for $\sigma(\ThetaAvg(t))$.

We can then use equation \eqref{eq:uniform-martingale} and Lemma \ref{lem:condition} followed by Lemma \ref{lem:limit-cond} to obtain that
\[
A(\zeta(t))-A(\zeta(0))-\int_0^t\cL_s A(\zeta(s)) ds
\]
is a martingale under any accumulation point $\widetilde\bP$ of the measures $\widetilde\bP^\ve_\ell$ with respect to the filtration $\cF_t$ with $\widetilde\bP(\{\zeta(0)=0\})=1$. In Proposition \ref{lem:unique} we will prove that such a problem has a unique solution thereby showing that $\lim_{\ve\to 0}\widetilde\bP^\ve_\ell=\widetilde\bP$.

Note that the time dependent operator $\cL_s$ is a second order operator, this means that the accumulation points of $\zeta_\ve$ do not satisfy a deterministic equation, but rather a stochastic one. Indeed our last task is to show that $\widetilde\bP$ is equal in law to the measure determined by the stochastic differential equation
 \begin{equation}\label{eq:final-dev}
 \begin{split}
 &d\zeta= \langle D\bar\omega\circ \ThetaAvg(t),\zeta\rangle dt+ \sigma dB\\
 &\zeta(0)=0
 \end{split}
 \end{equation}
 where $B$ is a standard $\bR^d$ dimensional Brownian motion. Note that the above equation is well defined in
 consequence of Lemma \ref{lem:coboundary} which shows that the matrix $\sigma^2$ is symmetric and non negative,
 hence $\sigma=\sigma^T$ is well defined and strictly positive if $\hat\omega$ is not a coboundary (see  Lemma \ref{lem:coboundary}). To conclude it suffices to show that the probability measure describing the
 solution of \eqref{eq:final-dev} satisfies the martingale problem.\footnote{ We do not prove that such a solution
   exists as this is a standard result in probability, \cite{V2}.} It follows from It\=o's calculus, indeed if $\zeta$
 is the solution of \eqref{eq:final-dev} and $A\in\cC^r$, then It\=o's formula reads
 \[
 d A(\zeta)=\sum_i\partial_{\zeta_i} A(\zeta) d\zeta_i+\frac 12\sum_{i,j,k}\partial_{\zeta_i} \partial_{\zeta_j} A(\zeta)\sigma_{ik}\sigma_{jk}dt.
\]
Integrating it from $s$ to $t$ and taking the conditional expectation we have
\[
\bE\left(A(\zeta(t))-A(\zeta(s))-\int_s^t\cL_\tau A(\zeta(\tau))d\tau\;\big|\; \cF_s\right)=0.
\]
See Appendix \ref{sec:moi} for more details on the relation between the Martingale problem and the theory of SDE and how this allows to dispense form It\=o's formula altogether.

We have thus seen that the measure determined by \eqref{eq:final-dev} satisfies the martingale problem, hence it must agree with $\widetilde\bP$ since $\widetilde\bP$ is the unique solution of the martingale problem. The proof of the Theorem is concluded by noticing that \eqref{eq:final-dev} defines a zero mean Gaussian process (see the end of the proof of Proposition \ref{lem:unique}).
\end{proof}
\subsection{Tightness}
\begin{prop}\label{lem:tight1}
  For every standard pair $\ell$, the measures $\{\widetilde\bP_\ell^\ve\}_{\ve>0}$ are tight.
\end{prop}
\begin{proof}
  Now the proof of tightness is less obvious since the paths have a Lipschitz constant that explodes.  Luckily, there
  exists a convenient criterion for tightness: Kolmogorov criterion \cite[Remark A.5]{V2}.
\begin{thm}[Kolmogorov]
Given a sequence of measures $\bP^\ve$ on $\cC^0([0,T],\bR)$, if there exists $\alpha,\beta, C>0$ such that
\[
\bE^\ve(|z(t)-z(s)|^{\beta})\leq C|t-s|^{1+\alpha}
\]
for all $t,s\in [0,T]$ and the distribution of $z(0)$ is tight, then $\{\bP^\ve\}$ is tight.
\end{thm}
Note that $\zeta_\ve(0)=0$.
Of course, it is easier to apply the above criteria with $\beta\in\bN$.  It is reasonable to expect that the fluctuations behave like a Brownian motion, so the variance should be finite.  To verify this let us compute first the case $\beta=2$.
Note that, setting $\hat\omega(x,\theta)=\omega(x,\theta)-\bar\omega(\theta)$,
\begin{equation}\label{eq:zeta-shape}
\begin{split}
\zeta_\ve(t)&=\sqrt{\ve}\sum_{k=0}^{\lceil \ve^{-1}t\rceil-1}\left[\omega(x_k,\theta_k)-\bar\omega(\ThetaAvg(\ve k))\right]+\cO(\sqrt \ve)\\
&=\sqrt{\ve}\sum_{k=0}^{\lceil \ve^{-1}t\rceil-1}\left[\hat\omega(x_k,\theta_k)+\bar\omega(\theta_k)-\bar\omega(\ThetaAvg(\ve k))\right]+\cO(\sqrt \ve)\\
&=\sqrt{\ve}\sum_{k=0}^{\lceil \ve^{-1}t\rceil-1}\left[\hat\omega(x_k,\theta_k)+\sqrt{\ve}D\bar\omega(\ThetaAvg(\ve k))\zeta_\ve(k\ve)\right]+\cO(\sqrt \ve)\\
&\quad\quad+\sum_{k=0}^{\lceil \ve^{-1}t\rceil-1}\cO(\ve^{\frac 32}\|\zeta_\ve(\ve k)\|^2).
\end{split}
\end{equation}
We start with a basic result.
\begin{lem}\label{lem:variance-basic}
For each standard pair $\ell$ and $k,l\in\{0,\dots,\ve^{-1}\}$, $k\geq l$, we have
\[
\mu_\ell\left(\left\|\sum_{j=l}^k\hat\omega(x_j,\theta_j)\right\|^2\right)\leq C_\# (k-l).
\]
\end{lem}
The proof of the above Lemma is postponed to the end of the section.  Let us see how it can be profitably used.
Note that, for $t=\ve k, s=\ve l$,
\begin{equation}\label{eq:variance-pre}
\widetilde\bE_\ell^\ve(\|\zeta(t)-\zeta(s)\|^2)\leq C_\#|t-s|+C_\#|t-s|\ve\sum_{j=l}^k\mu_\ell(\|\zeta_\ve(\ve j)\|^2)+C_\#\ve,
\end{equation}
where we have used Lemma \ref{lem:variance-basic} and the trivial estimate $\|\zeta_\ve\|\leq\Const\ve^{-\frac 12}$.
If we use the above with $s=0$ and define $M_k=\sup_{j\leq k}\mu_\ell(|\zeta_\ve(\ve j)|^2)$ we have
\[
M_k\leq C_\#\ve k+C_\#k^2\ve^2 M_k.
\]
Thus there exists $C>0$ such that, if $k\leq C\ve^{-1}$, we have $M_k\leq C_\#\ve k$.  Hence, we can substitute such an estimate in \eqref{eq:variance-pre} and obtain
\begin{equation}\label{eq:momentum2}
\widetilde\bE_\ell^\ve(\|\zeta(t)-\zeta(s)\|^2)\leq C_\#|t-s|+C_\#\ve.
\end{equation}
Since the estimate for $|t-s|\leq C_\#\ve$ is trivial, we have the bound,
\[
\widetilde\bE_\ell^\ve(\|\zeta(t)-\zeta(s)\|^2)\leq C_\#|t-s|.
\]
This is interesting but, unfortunately, it does not
suffice to apply the Kolmogorov criteria.  The next step could be to compute for $\beta=3$.  This has the well known
disadvantage of being an odd function of the path, and hence one has to deal with the absolute value. Due to this, it
turns out to be more convenient to consider directly the case $\beta=4$. This can be done in complete analogy with the
above computation, by first generalizing the result of Lemma \ref{lem:variance-basic} to higher momenta. Doing so we obtain
\begin{equation}\label{eq:momentum4}
\widetilde\bE_\ell^\ve(\|\zeta(t)-\zeta(s)\|^4)\leq C_\#|t-s|^{2},
\end{equation}
which concludes the proof of the proposition. Indeed, the proof of Lemma \ref{lem:variance-basic} explains how to treat correlations. Multiple correlations can be treated similarly and one can thus show that they do not contribute to the leading term. Thus the computation becomes similar (although much more involved) to the case of the sum independent zero mean random variables $X_i$ (where no correlations are present), that is
\[
\bE([\sum_{i=l}^kX_i]^4)=\sum_{i_1,\dots, i_4=l}^k\bE(X_{i_1}X_{i_2}X_{i_2}X_{i_4})=\sum_{i,j=l}^k\bE(X_i^2X_j^2)=\cO((k-l)^2).
\]

For future use let us record that, by equation \eqref{eq:momentum4} and
the Young inequality,
\begin{equation}\label{eq:momentum3}
\widetilde\bE_\ell^\ve(\|\zeta(t)-\zeta(s)\|^3)\leq C_\#|t-s|^{\frac 32}.\qedhere
\end{equation}
\end{proof}
We still owe the reader the
\begin{proof}[Proof of Lemma \ref{lem:variance-basic}]
The proof starts with a direct computation:\footnote{ To simplify notation we do the computation in the case $d=1$, the general case is identical.}
\begin{align*}
\mu_\ell\left(\left|\sum_{j=l}^k\hat\omega(x_j,\theta_j)\right|^2\right)
&\leq \sum_{j=l}^k\mu_\ell\left(\hat\omega(x_j,\theta_j)^2\right)\\
&\phantom{\leq}+2\sum_{j=l}^k\sum_{r=l+1}^k\mu_\ell\left(\hat\omega(x_j,\theta_j)\hat\omega(x_r,\theta_r)\right)\\
&\leq C_\#|k-l|+2\sum_{j=l}^k\sum_{r=j+1}^k\mu_\ell\left(\hat\omega(x_j,\theta_j)\hat\omega(x_r,\theta_r)\right).
\end{align*}
To compute the last correlation, remember Proposition \ref{p_invarianceStandardPairs}.  We can thus call $\frkL_j$ the standard family associated to $(F_\ve^j)_*\mu_\ell$ and, for $r\geq j$, we write
\[
\begin{split}
\mu_\ell\left(\hat\omega(x_j,\theta_j)\hat\omega(x_r,\theta_r)\right)&=\sum_{\ell_1\in\frkL_j}\nu_{\ell_1}\mu_{\ell_1}\left(\hat\omega(x_0,\theta_0)\hat\omega(x_{r-j},\theta_{r-j})\right)\\
&=\sum_{\ell_1\in\frkL_j}\nu_{\ell_1}\int_{a_{\ell_1}}^{b_{\ell_1}}\rho_{\ell_1}(x)\hat\omega(x,G_{\ell_1}(x))\hat\omega(x_{r-j},\theta_{r-j}).
\end{split}
\]
We would like to argue as in the proof of Lemma \ref{lem:average} and try to reduce the problem to
\[
\begin{split}
\int_{a_{\ell_1}}^{b_{\ell_1}}\rho_{\ell_1}(x)\hat\omega(x,\theta^*_{\ell_1})\hat\omega(x_{r-j},\theta^*_{\ell_1})&=\int_{a_{\ell_1}}^{b_{\ell_1}}\rho_{\ell_1}(x)\hat\omega(x,\theta^*_{\ell_1})\hat\omega(f_*^{r-j}(Y_{r-j}(x)),\theta^*_{\ell_1})\\
&=
\int_{\bT^1}\tilde\rho(x)\hat\omega(Y^{-1}_{r-j}(x),\theta^*_{\ell_1})\hat\omega(f_*^{r-j}(x),\theta^*_{\ell_1}),
\end{split}
\]
but then the mistake that we would make substituting $\tilde \rho$ with $\bar\rho$ is too big for our current purposes.  It is thus necessary to be more subtle.  The idea is to write $\rho_{\ell_1}(x)\hat\omega(x,G_{\ell_1}(x))=\alpha_1\hat \rho_1(x)+\alpha_2\hat \rho_2(x)$, where $\hat\rho_1,\hat \rho_2$ are standard densities.\footnote{ In fact, it would be more convenient to define standard pairs with signed (actually even complex) measures, but let us keep it simple.} Note that $\alpha_1,\alpha_2$ are uniformly bounded.  Next, let us fix $L>0$ to be chosen later and assume $r-j\geq L$.  Since $\ell_{1,i}=(G,\hat\rho_i)$ are standard pairs, by construction, calling $\frkL_{\ell_{1,i}}=(F^{r-j-L})_*\mu_{\ell_{1,i}}$ we have
\[
\begin{split}
\int_{a_{\ell_1}}^{b_{\ell_1}}\hat \rho_i(x)\hat\omega(x_{r-j},\theta_{r-j})&=\sum_{\ell_2\in \frkL_{\ell_{1,i}}}\nu_{\ell_2}\int_{a_{\ell_2}}^{b_{\ell_2}}\rho_{\ell_{2}}(x)\hat\omega(f_*^{L}(Y_L(x)),\theta^*_{\ell_2})+\cO(\ve L)\\
&=\sum_{\ell_2\in \frkL_{\ell_{1,i}}}\nu_{\ell_2}\int_{\bT^1}\tilde\rho(x)\hat\omega(f_*^{L}(x),\theta^*_{\ell_2})+\cO(\ve L)\\
&=\sum_{\ell_2\in \frkL_{\ell_{1,i}}}\nu_{\ell_2}\int_{\bT^1}\bar\rho(x)\hat\omega(f_*^{L}(x),\theta^*_{\ell_2})+\cO(\ve L^2)\\
&=\cO(e^{-c_\# L}+\ve L^2),
\end{split}
\]
due to the decay of correlations for the map $f_*$ and the fact that $\hat\omega(\cdot,\theta^*_{\ell_2})$ is a zero average function for the invariant measure of $f_*$.  By the above we have
\[
\mu_\ell\left(\left|\sum_{j=l}^k\hat\omega(x_j,\theta_j)\right|^2\right)
\leq C_\# \sum_{j=l}^k\left\{[e^{-c_\# L}+\ve L^2](k-j) +1+\ve L^3\right\}
\]
which yields the result by choosing $L=c\log (k-j)$ for $c$ large enough.
\end{proof}

%%%%%%%%%%%%%%%%%%%%%%%%%%%%%%%%%%%%
\subsection{Differentiating with respect to time (poor man's It\=o's formula)}
\begin{prop} \label{prop:diff-eq} For every standard pair $\ell$ and $A\in\cC^3(S,\bR)$ we have
\[
\lim_{\ve\to 0}\sup_{\ell\in\overline\frkL_\ve}\left|\widetilde\bE_\ell^\ve\left(A(\zeta(t))-A(\zeta(0))-\int_0^t\cL_s A(\zeta(s)) ds\right)\right|=0.
\]
\end{prop}
\begin{proof}
As in Lemma \ref{lem:average}, the idea is to fix $h\in(0,1)$ to be chosen later, and compute
\begin{equation}\label{eq:two-step-one}
\begin{split}
\widetilde\bE^\ve_\ell&(A(\zeta(t+h))-A(\zeta(t)))=\widetilde\bE^\ve_\ell(\langle\nabla A(\zeta(t)),\zeta(t+h)-\zeta(t)\rangle)\\
&+\widetilde\bE^\ve_\ell(\frac 12\langle (D^2A)(\zeta(t))(\zeta(t+h)-\zeta(t)), \zeta(t+h)-\zeta(t)\rangle)+\cO( h^{\frac 32}),
\end{split}
\end{equation}
where we have used \eqref{eq:momentum3}.  Unfortunately this time the computation is a bit lengthy and rather boring, yet it basically does not contain any new idea, it is just a brute force computation.

Let us start computing the last term of \eqref{eq:two-step-one}.
Setting $\zeta^h(t)=\zeta(t+h)-\zeta(t)$ and $\Omega^h=\sum_{k=t\ve^{-1}}^{(t+h)\ve^{-1}}\hat\omega(x_k,\theta_k)$, by equations \eqref{eq:zeta-shape} and using the trivial estimate $\|\zeta_\ve(t)\|\leq \Const \ve^{-\frac12}$, we have
\[
\begin{split}
\widetilde\bE^\ve_\ell&(\langle (D^2A)(\zeta(t))\zeta^h(t), \zeta^h(t)\rangle)=
\ve\sum_{k,j=t\ve^{-1}}^{(t+h)\ve^{-1}}\mu_\ell(\langle (D^2A)(\zeta_\ve(t))\hat\omega(x_k,\theta_k),\hat\omega(x_j,\theta_j)\rangle)\\
&+\cO\left(\ve^{\frac 32}\sum_{j=\ve^{-1}t}^{(t+h)\ve^{-1}}\mu_\ell\left(\left\|\Omega^h\right\|\,\|\zeta_\ve(j\ve)\|\right)\right)+\cO\left(\ve\mu_\ell(\|\Omega^h\|)\right)\\
&+\cO\left(\ve^{2}\sum_{j=t\ve^{-1}}^{(t+h)\ve^{-1}}\mu_\ell(\|\Omega^h\|\,\|\zeta_\ve(j\ve)\|^{2})+\ve^2\sum_{k,j=t\ve^{-1}}^{(t+h)\ve^{-1}}\mu_\ell(\|\zeta_\ve(k\ve)\|\, \|\zeta_\ve(j\ve)\|)\right)\\
&+\cO\left(\ve^{\frac 32}\sum_{k=t\ve^{-1}}^{(t+h)\ve^{-1}}\mu_\ell(\|\zeta_\ve(k\ve)\|)+\ve^{\frac 52}\sum_{k,j=t\ve^{-1}}^{(t+h)\ve^{-1}}\mu_\ell(\|\zeta_\ve(k\ve)\|\, \|\zeta_\ve(j\ve)\|^{2})+\ve\right)\\
&+\cO\left(\ve^{2}\sum_{k=t\ve^{-1}}^{(t+h)\ve^{-1}}\mu_\ell(\|\zeta_\ve(k\ve)\|^2)+\ve^{3}\sum_{k,j=t\ve^{-1}}^{(t+h)\ve^{-1}}\mu_\ell(\|\zeta_\ve(k\ve)\|^{2}\, \|\zeta_\ve(j\ve)\|^{2})\right).
\end{split}
\]
Observe that~\eqref{eq:momentum2},~\eqref{eq:momentum3} and~\eqref{eq:momentum4} yield
\[
\mu_\ell(\|\zeta_\ve(k\ve)\|^m)= \mu_\ell(\|\zeta_\ve(k\ve)-\zeta_\ve(0)\|^m)\leq \Const (\ve k)^{\frac m2}\leq \Const
\]
for $m\in\{1,2,3,4\}$ and $k\leq \Const \ve^{-1}$.
We can now use Lemma \ref{lem:variance-basic} together with Schwartz inequality to obtain
\begin{equation}\label{eq:first-var}
\begin{split}
\widetilde\bE^\ve_\ell(\langle (D^2A)(\zeta(t))\zeta^h(t), \zeta^h(t)\rangle)=&
\,\ve\hskip-.3cm\sum_{k,j=t\ve^{-1}}^{(t+h)\ve^{-1}}\hskip-.3cm\mu_\ell(\langle (D^2A)(\zeta_\ve(t))\hat\omega(x_k,\theta_k),\hat\omega(x_j,\theta_j)\rangle)\\
&+\cO(\sqrt{\ve h}+ h^2+\ve).
\end{split}
\end{equation}
Next, we must perform a similar analysis on the first term of equation \eqref{eq:two-step-one}.
\begin{equation}\label{eq:palla-uno}
\begin{split}
\widetilde\bE^\ve_\ell(\langle\nabla A(\zeta(t))&,\zeta^h(t)\rangle)=\sqrt\ve\sum_{k=t\ve^{-1}}^{(t+h)\ve^{-1}}\mu_\ell(\langle\nabla A(\zeta_\ve(t)),\hat \omega(x_k,\theta_k)\rangle)\\
&+\ve\sum_{k=t\ve^{-1}}^{(t+h)\ve^{-1}}\mu_\ell(\langle\nabla A(\zeta_\ve(t)),D\bar\omega(\ThetaAvg(\ve k))\zeta_\ve(\ve k)\rangle)+\cO(\sqrt\ve).
\end{split}
\end{equation}
To estimate the term in the second line of \eqref{eq:palla-uno} we have to use again \eqref{eq:zeta-shape}:
\[
\begin{split}
  &\sum_{k=t\ve^{-1}}^{(t+h)\ve^{-1}}\mu_\ell(\langle\nabla A(\zeta_\ve(t)),D\bar\omega(\ThetaAvg(\ve k))\zeta_\ve(\ve
  k)\rangle)=
  h\ve^{-1}\mu_\ell(\langle\nabla A(\zeta_\ve(t)),D\bar\omega(\ThetaAvg(t))\zeta_\ve(t)\rangle)\\
  &\quad+\cO(\ve^{-1} h^2+\ve^{-\frac 12}
  h)+\sqrt\ve\sum_{k=t\ve^{-1}}^{(t+h)\ve^{-1}}\sum_{j=t\ve^{-1}}^k\mu_\ell(\langle\nabla
  A(\zeta_\ve(t)),D\bar\omega(\ThetaAvg(t))\hat\omega(x_j,\theta_j)\rangle).
\end{split}
\]
To compute the last term in the above equation let $\frkL_\ell$ be the standard family generated by $\ell$ at time $\ve^{-1} t$, then, setting $\alpha_\ve (\theta, t)= \nabla A(\ve^{-\frac 12}(\theta-\ThetaAvg(t)))$ and $\hat j=j-t\ve^{-1}$, we can write
\[
\mu_\ell(\langle\nabla A(\zeta_\ve(t)),D\bar\omega(\ThetaAvg(t))\hat\omega(x_j,\theta_j)\rangle)=\sum_{\ell_1\in \frkL_\ell}\sum_{r,s=1}^d\nu_{\ell_1}\mu_{\ell_1}(\alpha_\ve(\theta_0,t)_r D\bar\omega(\ThetaAvg(t))_{r,s}\hat\omega(x_{\hat j},\theta_{\hat j})_s).
\]
Next, notice that for every $r$, the signed measure $\mu_{\ell_1,r}(\phi)=\mu_{\ell_1}(\alpha_\ve(\theta_0,t)_r\phi)$
has density $\rho_{\ell_1}\alpha_\ve(G_{\ell_1}(x),t)_r$ whose derivative is uniformly bounded in $\ve, t$. We can then
write $\mu_{\ell_1,r}$ as a linear combination of two standard pairs $\ell_{1,i}$. Finally, given $L\in\bN$, if $\hat
j\geq L$, we can consider the standard families $\cL_{\ell_{1,i}}$ generated by $\ell_{1,i}$ at time $\hat j-L$ and
write, arguing as in the proof of Lemma \ref{lem:variance-basic},
\[
\begin{split}
\mu_{\ell_{1,i}}(\hat\omega(x_{\hat j},\theta_{\hat j})_s)&=\sum_{\ell_2\in \cL_{\ell_{1,i}}}\nu_{\ell_2}\mu_{\ell_2}(\hat\omega(x_{L},\theta_{L})_s)\\
&=\sum_{\ell_2\in \cL_{\ell_{1,i}}}\nu_{\ell_2}\int_{a_{\ell_2}}^{b_{\ell_2}}\rho_{\ell_2}(x)\hat\omega(f_{\theta_{\ell_2}^*}^L(x),\theta_{\ell_2}^*)_s+\cO(\ve L^2)=\cO(e^{-\Const L}+\ve L^2).
\end{split}
\]
Collecting all the above estimates yields
\begin{equation}\label{eq:palla-due}
\begin{split}
&\ve\sum_{k=t\ve^{-1}}^{(t+h)\ve^{-1}}\mu_\ell(\langle\nabla A(\zeta_\ve(t)),D\bar\omega(\ThetaAvg(\ve k))\zeta_\ve(\ve k)\rangle)=\cO( h^2+\ve^{\frac 12} h)\\
&+h\mu_\ell(\langle\nabla A(\zeta_\ve(t)),D\bar\omega(\ThetaAvg(t))\zeta_\ve(t)\rangle)+\cO(h^2 \ve^{\frac 12}L^2+h^2\ve^{-\frac 12}e^{-\Const L}+\ve^{\frac 12}L h).
\end{split}
\end{equation}
To deal with the second term in the first line of equation~\eqref{eq:palla-uno} we argue as before:
\[
\begin{split}
\sum_{k=t\ve^{-1}}^{(t+h)\ve^{-1}}\mu_\ell(\langle\nabla A(\zeta_\ve(t)),\hat \omega(x_k,\theta_k)\rangle)&=\sum_{k=t\ve^{-1}}^{t\ve^{-1}+L}\mu_\ell(\langle\nabla A(\zeta_\ve(t)),\hat \omega(x_k,\theta_k)\rangle)\\
&\quad\quad+\cO(h L^2+\ve^{-1}he^{\Const L})\\
&=\cO(L+h L^2+\ve^{-1}he^{\Const L}).
\end{split}
\]
Collecting the above computations and remembering \eqref{eq:zeta-shape} we obtain
\begin{equation}\label{eq:second-var}
\begin{split}
\widetilde\bE^\ve_\ell(\langle\nabla A(\zeta(t)),\zeta^h(t)\rangle)=&h\widetilde\bE^\ve_\ell(\langle\nabla A(\zeta(t)),D\bar\omega(\ThetaAvg(t))\zeta(t)\rangle)\\
&+\cO(h^2+L\sqrt\ve+h\sqrt \ve L^2)
\end{split}
\end{equation}
provided $L$ is chosen in the interval $[C_*\ln\ve^{-1},\ve^{-\frac 14}]$ with $C_*>0$ large enough.

To conclude we must compute the term on the right hand side of the first line of equation \eqref{eq:first-var}.  Consider first the case $|j-k|>L$.  Suppose $k>j$, the other case being equal, then, letting $\frkL_\ell$ be the standard family generated by $\ell$ at time $\ve^{-1} t$, and set $\hat k=k-\ve^{-1}t, \hat j=j-\ve^{-1}t$, $B(x,\theta,t)=(D^2A)(\ve^{-\frac 12}(\theta-\ThetaAvg(t)))$
\[
\mu_\ell(\langle (D^2A)(\zeta_\ve(t))\hat\omega(x_k,\theta_k),\hat\omega(x_j,\theta_j)\rangle)=\sum_{\ell_1\in\frkL_\ell}\nu_{\ell_1}\mu_{\ell_1}(\langle B(x_0,\theta_0,t)\hat\omega(x_{\hat k},\theta_{\hat k}),\hat\omega(x_{\hat j},\theta_{\hat j})\rangle).
\]
Note that the signed measure $\hat\mu_{\ell_1,r,s}(g)=\mu_{\ell_1}(B_{r,s}g)$ has a density with uniformly bounded derivative given by $\hat\rho_{\ell_1,r,s}=\rho_{\ell_1}B(x,G_{\ell_1}(x),t)_{r,s}$. Such a density can then be written as a linear combination of standard densities $\hat\rho_{\ell_1,r,s}=\alpha_{1,\ell_1,r,s}\rho_{1,\ell_1,r,s}+\alpha_{2,\ell_1,r,s}\rho_{2,\ell_1,r,s}$ with uniformly bounded coefficients $\alpha_{i,\ell_1,r,s}$.  We can then use the same trick at time $j$ and then at time $k-L$ and obtain that the quantity we are interested in can be written as a linear combination of quantities of the type
\[
\begin{split}
\mu_{\ell_3,r,s}(\hat\omega_s(x_{L},\theta_{L}))&=\mu_{\ell_3,r,s}(\hat\omega_s(x_{L},\theta_{\ell_3}^*)+\cO(L\ve)
=\int_a^b\tilde \rho_{r,s} \hat\omega_s(f_{\theta_{\ell_3}^*}^L( x),\theta_{\ell_3}^*)+\cO(L^2\ve)\\
&=\cO(e^{-\Const L}+L^2\ve)
\end{split}
\]
where we argued as in the proof of Lemma \ref{lem:variance-basic}.  Thus the total contribution of all such terms is of order $L^2 h^2+\ve^{-1}e^{-\Const L}h^2$.  Next, the terms such that $|k-j|\leq L$ but $j\leq \ve^{-1}t +L$ give a total contribution of order $L^2\ve$ while to estimate the other terms it is convenient to proceed as before but stop at the time $j-L$. Setting $\tilde k=k-j+L$ we obtain terms of the form
\[
\begin{split}
\mu_{\ell_2,r,s}( \hat\omega_s(x_{\tilde k},\theta_{\tilde k})\hat\omega_r ( x_{L},\theta_{L})\rangle)= \Gamma_{k-j}(\theta_{\ell_2}^*)_{r,s}+\cO(e^{-\Const L}+L^2\ve)
\end{split}
\]
where
\[
\Gamma_{k}(\theta)= \int_S\hat\omega(f_{\theta}^{k}(x),\theta)\otimes \hat\omega(x,\theta)\;m_{\theta}(dx).
\]
The case $j>k$ yields the same results but with $\Gamma_j^*$.
Remembering the smooth dependence of the covariance on the parameter $\theta$ (see \cite{L1}), substituting the result
of the above computation in \eqref{eq:first-var} and then \eqref{eq:first-var} and \eqref{eq:second-var} in
\eqref{eq:two-step-one} we finally have
\[
\begin{split}
&\widetilde\bE^\ve_\ell(A(\zeta(t+h))-A(\zeta(t)))=h\widetilde\bE^\ve_\ell(\langle\nabla A(\zeta(t)), D\bar\omega(\ThetaAvg(t))\zeta(t)\rangle)\\
&\quad+h\widetilde\bE^\ve_\ell(\operatorname{Tr}( \sigma^2(\ThetaAvg(t))D^2A(\zeta(t))))+\cO(L\sqrt\ve+hL^2\sqrt \ve+h^2 L^2)\\
&=\int_t^{t+h}\left[\widetilde\bE^\ve_\ell(\langle\nabla A(\zeta(s)), D\bar\omega(\ThetaAvg(s))\zeta(s)\rangle)+\widetilde\bE^\ve_\ell(\operatorname{Tr}( \sigma^2(\ThetaAvg(s))D^2A(\zeta(s))) \right]ds\\
&\quad+\cO(L\sqrt\ve+hL^2\sqrt \ve+h^{\frac 32}+h^2 L^2).
\end{split}
\]
The proposition follows by summing the $h^{-1}t$ terms in the interval $[0,t]$ and by choosing $L=\ve^{-\frac 1{100}}$
and $h=\ve^{\frac 13}$.
\end{proof}

In the previous Lemma the expression $\sigma^2$ just stands for a specified matrix, we did not prove that such a matrix
is positive definite and hence it has a well defined square root $\sigma$, nor we have much understanding of the
properties of such a $\sigma$ (provided it exists). To clarify this is our next task.
\begin{lem}\label{lem:coboundary} The matrices $\sigma^2(s)$, $s\in [0,T]$, are symmetric and non negative, hence they have a unique real symmetric square root $\sigma(s)$. In addition, if, for each $v\in\bR^d$, $\langle v,\bar\omega\rangle$ is not a smooth coboundary, then there exists $c>0$ such that $\sigma(s)\geq c\Id$.
\end{lem}
\begin{proof}
For each $v\in\bR^d$ a direct computation shows that
\[
\begin{split}
&\lim_{n\to\infty}\frac1n m_\theta\left(\left[\sum_{k=0}^{n-1}\langle v,\omega(f_\theta^k(\cdot),\theta)\rangle\right]^2\right)=\lim_{n\to\infty}\frac1n\sum_{k,j=0}^{n-1}m_\theta\left(\langle v,\omega(f_\theta^k(\cdot),\theta)\rangle \langle v,\omega(f_\theta^j(\cdot),\theta)\rangle\right)\\
&=m_\theta(\langle v,\omega(\cdot,\theta)\rangle^2)+\lim_{n\to\infty}\frac 2n\sum_{k=1}^{n-1} (n-k)m_\theta(\langle \omega(\cdot,\theta),v\rangle \langle v,\omega(f^k_\theta(\cdot),\theta)\rangle)\\
&=m_\theta(\langle v,\omega(\cdot,\theta)\rangle^2)+2\sum_{k=1}^{\infty} m_\theta\left(\langle \omega(\cdot,\theta),v\rangle \langle v,\omega(f^k_\theta(\cdot),\theta)\rangle\right)=\langle v,\sigma(\theta)^2v\rangle.
\end{split}
\]
This implies that $\sigma(\theta)^2\geq 0$ and since it is symmetric, there exists, unique, $\sigma(\theta)$ symmetric and non-negative.
On the other hand if $\langle v,\sigma^2(\theta)v\rangle=0$, then, by the decay of correlations and the above equation, we have
\[
\begin{split}
m_\theta\left(\left[\sum_{k=1}^{n-1}\langle v,\omega(f_\theta^k(\cdot),\theta)\rangle\right]^2\right)&= n\, m_\theta(\langle v,\omega(\cdot,\theta)\rangle^2)\\
&\quad+2n\sum_{k=0}^{n-1}m_\theta\left(\langle v,\omega(\cdot,\theta)\rangle \langle v,\omega(f_\theta^k(\cdot),\theta)\rangle\right)+\cO(1)\\
&=2n\sum_{k=n}^{\infty}m_\theta\left(\langle v,\omega(\cdot,\theta)\rangle \langle v,\omega(f_\theta^k(\cdot),\theta)\rangle\right)+\cO(1)=\cO(1).
\end{split}
\]
Thus the $L^2$ norm of $\phi_n=\sum_{k=1}^{n-1}\langle v,\omega(f_\theta^k(\cdot),\theta)\rangle$ is uniformly bounded. Hence there exist a weakly convergent subsequence. Let $\phi\in L^2$ be an accumulation point, then for each $\vf\in\cC^1$ we have
\[
m_\theta(\phi\circ f_\theta \vf)=\lim_{k\to\infty} m_\theta(\phi_{n_k}\circ f_\theta \vf)=m_\theta(\phi \vf)-m_\theta(\langle v,\omega(\cdot,\theta)\rangle\vf)
\]
That is $\langle v,\omega(x,\theta)\rangle=\phi(x)-\phi\circ f_\theta(x)$. In other words $\langle
v,\omega(x,\theta)\rangle$ is an $L^2$ coboundary. Since the Livsic Theorem \cite{livs} states that the solution of the
cohomological equation must be smooth, we have $\phi\in\cC^1$.
\end{proof}

\subsection{Uniqueness of the Martingale Problem}\

\noindent We are left with the task of proving the uniqueness of the martingale problem. Note that in the present case the operator depends explicitly on time. Thus if we want to set the initial condition at a time $s\neq 0$ we need to slightly generalise the definition of martingale problem. To avoid this, for simplicity, here we consider only initial conditions at time zero, which suffice for our purposes. In fact, we will consider only the initial condition $\zeta(0)=0$, since it is the only one we are interested in. We have then the same definition of the martingale problem as in Definition \ref{def:martingale}, apart form the fact that $\cL$ is replaced by $\cL_s$ and $y=0$.

Since the operators $\cL_s$ are second order operators, we could use well known results. Indeed, there exists a deep
theory due do Stroock and Varadhan that establishes the uniqueness of the martingale problem for a wide class of second
order operators,~\cite{SV}. Yet, our case is especially simple because the coefficients of the higher order part of the
differential operator depend only on time and not on $\zeta$. In this case it is possible to modify a simple proof of
the uniqueness that works when all the coefficients depend only on time,~\cite[Lemma 6.1.4]{SV}. We provide here the
argument for the reader's convenience.
\begin{prop}\label{lem:unique}
The martingale problem associated to the operators $\cL_s$ in Proposition \ref{prop:diff-eq} has a unique solution.
\end{prop}
\begin{proof}
As already noticed, $\cL_t$, defined in \eqref{eq:secondorderop}, depends on $\zeta$ only via the coefficient of the first order part. It is then natural to try to change measure so that such a dependence is eliminated and we obtain a martingale problem with respect to an operator with all coefficients depending only on time, then one can conclude arguing as in  \cite[Lemma 6.1.4]{SV}. Such a reduction is routinely done in probability via the Cameron-Martin-Girsanov formula. Yet, given the simple situation at hand one can proceed in a much more naive manner. Let $S(t):[0,T]\to M_d$, $M_d$ being the space of $d\times d$ matrices, be the generated by the differential equation
\[
\begin{split}
&\dot S(t)=-D\overline\omega(\ThetaAvg(t)) S(t)\\
& S(0)=\Id.
\end{split}
\]
Note that, setting $\varsigma(t)=\det S(t)$ and $B(t)=D\overline\omega(\ThetaAvg(t))$, we have
\[
\begin{split}
&\dot\varsigma(t)=-\operatorname{tr}(B(t))\varsigma(t)\\
& \varsigma(0)=1.
\end{split}
\]
The above implies that $S(t)$ is invertible.

Define the map $\cS\in \cC^0(\cC^0([0,T],\bR^d), \cC^0([0,T],\bR^d))$ by $[\cS\zeta](t)=S(t)\zeta(t)$ and set $\overline\bP=\cS_*\widetilde\bP$.  Note that the map $\cS$ is invertible.
Finally, we define the operator
\[
\widehat\cL_t=\frac 12\sum_{i,j}[\widehat\Sigma(t)^2]_{i,j}\partial_{\zeta_i}\partial_{\zeta_j},
\]
where $\widehat\Sigma(t)^2=S(t)\sigma(t)^2S(t)^*$, $\sigma(t)=\sigma( \ThetaAvg(t))$ as mentioned after \eqref{eq:secondorderop}.
Let us verify that $\overline\bP$ satisfies the martingale problem with respect to the operators $\widehat\cL_t$.  By Lemma \ref{lem:basic-m} we have
\[
\begin{split}
\frac{d}{dt}\overline\bE(A(\zeta(t))\;|\;\cF_s)&=\frac{d}{dt}\widetilde\bE(A(S(t)\zeta(t))\;|\;\cF_s)\\
&=\widetilde\bE(\dot S(t)\nabla A(S(t)\zeta(t))+\cL_t A(S(t)\zeta(t))\;|\;\cF_s)\\
&=\frac 12\widetilde\bE\left(\sum_{i,j,k,l}\sigma^2(t)_{i,j}\partial_{\zeta_k}\partial_{\zeta_l} A(S(t)\zeta(t)) S(t)_{k,i}S(t)_{l,j}\;\bigg|\;\cF_s\right)\\
&=\overline\bE(\widehat\cL_t A(\zeta(t))\;|\;\cF_s).
\end{split}
\]
Thus the claim follows by Lemma \ref{lem:basic-m} again.

Accordingly, if we prove that the above martingale problem has a unique solution, then $\overline\bP$ is uniquely determined, which, in turn, determines uniquely $\widetilde\bP$, concluding the proof.

Let us define the function $B\in\cC^1(\bR^{2d+1},\bR)$ by
\[
B(t,\zeta,\lambda)=e^{\langle\lambda,\zeta\rangle-\frac 12\int_s^t\langle \lambda,\widehat\Sigma(\tau)^2\lambda\rangle d\tau}
\]
then Lemma \ref{lem:basic-m} implies
\[
\frac d{dt}\overline\bE(B(t,\zeta(t),\lambda)\;|\;\cF_s)=\overline\bE(-\frac 12\langle\lambda,\widehat\Sigma(t)^2\lambda\rangle B(t,\zeta(t),\lambda)+\widehat\cL_tB(t,\zeta(t),\lambda)\;|\;\cF_s)=0.
\]
Hence
\[
\overline\bE(e^{\langle\lambda,\zeta(t)\rangle}\;|\;\cF_s)=e^{\langle\lambda,\zeta(s)\rangle+\frac 12\int_s^t\langle \lambda,\widehat\Sigma(\tau)^2\lambda\rangle d\tau}.
\]
From this follows that the finite dimensional distributions are uniquely determined. Indeed, for each $n\in\bN$, $\{\lambda_i\}_{i=1}^n$ and $0\leq t_1<\dots<t_n$ we have
\[
\begin{split}
\overline\bE\left(e^{\sum_{i=1}^n\langle\lambda_i,\zeta(t_i)\rangle}\right)&=\overline\bE\left(e^{\sum_{i=1}^{n-1}\langle\lambda_i,\zeta(t_i)\rangle}\overline\bE\left(e^{\langle\lambda_n,\zeta(t_n)\rangle}\;\big|\;\cF_{t_{n-1}}\right)\right)\\
&=\overline\bE\left(e^{\sum_{i=1}^{n-2}\langle\lambda_i,\zeta(t_i)\rangle+\langle\lambda_{n-1}+\lambda_n,\zeta(t_{n-1})\rangle}\right)
e^{\frac 12\int_{t_{n-1}}^{t_n}\langle \lambda_n,\widehat\Sigma(\tau)^2\lambda_n\rangle d\tau}\\
&=e^{\frac 12\int_{0}^{t_n}\langle\sum_{i=n(\tau)}^{n} \lambda_i,\widehat\Sigma(\tau)^2\sum_{i=n(\tau)}^{n} \lambda_i\rangle d\tau}
\end{split}
\]
where $n(\tau)=\inf\{m\;|\; t_m\geq \tau\}$. This concludes the Lemma since it implies that the measure is uniquely determined on the sets that generate the $\sigma$-algebra.\footnote{ See the discussion at the beginning of Section \ref{sec:prelim}.} Note that we have also proven that the process is a zero mean Gaussian process; this, after translating back to the original measure, generalises Remark \ref{rem:gauss-sol}.
\end{proof}
\appendix
%%%%%%%%%%%%%%%%%%%%%%%%%%%%%%%%%%%%%%%%%%%%%%
%%%%%%%%%%%%%%%%%%%%%%%%%%%%%%%%%%%%%%%%%%%%%%
%% Geometry
%%%%%%%%%%%%%%%%%%%%%%%%%%%%%%%%%%%%%%%%%%%%%%
%%%%%%%%%%%%%%%%%%%%%%%%%%%%%%%%%%%%%%%%%%%%%%
\section{Geometry}
For $c>0$, consider the cones $\cC_c=\{(\xi,\eta)\in\bR^2\;:\; |\eta|\leq\ve c|\xi|\}$.  Note that
\[
\deh F_\ve=\begin{pmatrix} \partial_{x}f &\partial_\theta f\\ \ve\partial_x\omega &1+\ve\partial_
\theta\omega\end{pmatrix}.
\]
Thus if $(1,\ve u)\in\cC_c$,
\begin{align*}
  \deh_pF_\ve(1,\ve u) &= (\partial_{x}f(p) +\ve u\partial_\theta
  f(p),\ve\partial_x\omega(p) +\ve u+\ve^2 u\partial_\theta\omega(p))\\
  &=\partial_{x}f(p)\left(1 +\ve \frac{\partial_\theta
      f(p)}{\partial_{x}f(p)}u\right)\cdot (1,\ve \Xi_p(u))
\end{align*}
where
\begin{equation}\label{e_evolutionXi}
  \Xi_p(u)=\frac{\partial_x\omega(p) + (1+\ve\partial_\theta\omega(p))u}{\partial_{x}f(p)
    + \ve \partial_\theta f(p)u}.
\end{equation}
Thus the vector $(1,\ve u)$ is mapped to the vector $(1,\ve \Xi_p(u))$.
Thus letting  $K=\max\{\|\partial_x\omega\|_\infty, \|\partial_\theta\omega\|_\infty, \|\partial_\theta f\|_\infty\}$ we have, for $|u|\leq c$ and assuming $K\ve c\leq 1$,
\[
|\Xi_p(u)|\leq \frac {K+(1+\ve K)c}{\lambda-\ve Kc}\leq\frac{K+1+c}{\lambda-1}.
\]
Thus, if we choose $c\in[ \frac{K+1}{\lambda-2},(\ve K)^{-1}]$ we have that $\deh_pF_\ve(\cC_c)\subset\cC_c$. Since this implies that $\deh_pF_\ve^{-1}\complement\cC_c \subset \complement\cC_c$ we have that the complementary cone $\complement\cC_{K\vei}$ is invariant under $\deh F_{\ve}^{-1}$. From now on we fix $c=\frac{K+1}{\lambda-2}$.  

Hence, for any $p\in\bT^{1+d}$ and
$n\in\bN$, we can define the quantities $\nuold_n,u_n,\sigmaold_n,\muold_n$ as follows:
\begin{align}\label{e_defineSlopes}
  \deh_p F_{\ve}^n (1,0)&=\nuold_n ( 1, \ve u_n )&\deh_p F_{\ve}^n (\sigmaold_n ,1) &= \muold_n (0,1)
\end{align}
with $|u_n|\leq c$ and $|\sigmaold_n|\leq K$.  For each $n$ the slope field $\sigmaold_n$ is
smooth, therefore integrable; given any small $\Delta>0$ and $p=(x,\theta)\in\bT^{1+d}$,
define $\cW_n^\nt(p,\Delta)$ the \emph{local $n$-step central manifold of size $\Delta$}
as the connected component containing $p$ of the intersection with the strip
$\{|\theta'-\theta|<\Delta\}$ of the integral curve of $(\sigmaold_n,1)$ passing through $p$.

Notice that, by definition, $\deh_p F_\ve(\sigmaold_n(p),1) = \muold_n/\muold_{n-1}(\sigmaold_{n-1}(F_
\ve p),1)$; thus, by
definition, there exists a constant $\expb$ such that:
\begin{equation}\label{e_trivialBound}
  \expo{-\expb\ve}\leq \frac{\muold_n}{\muold_{n-1}} \leq \expo{\expb\ve}.
\end{equation}
Furthermore, define $\Gamma_n=\prod_{k=0}^{n-1}\partial_x f\circ F_\ve^k$, and let
\begin{equation}\label{e_eqa}
  \expa=c\left\|\frac{\partial_\theta f}{\partial_x f}\right\|_\infty.
\end{equation}
Clearly,
\begin{equation}\label{e_***}
\Gamma_n\expo{-\expa\ve n} \leq\nuold_n\leq \Gamma_n\expo{\expa\ve n}.
\end{equation}

\section{Shadowing}\label{sec:shadow}
In this section we provide a simple quantitative version of shadowing that is needed in the argument.  Let
$(x_k,\theta_k)=F_\ve^k(x,\theta)$ with $k\in\{0,\dots, n\}$.  We assume that $\theta$ belongs to the range of a
standard pair $\ell$ (i.e., $\theta=G(x)$ for some $x\in [a,b]$).

Let $\theta^*\in S$ such that $\|\theta^*-\theta\|\leq \ve$ and set $f_*(x)=f(x,\theta^*)$.
Let us denote with $\pi_x: X\to S$ the canonical projection on the $x$ coordinate;
then, for any $s\in [0,1]$, let
\begin{align*}
  H_{n}(x, z, s)&=\pi_x F_{s\ve}^n(x,\theta^*+s(G_\ell(x)-\theta^*))- f_*^n(z)
\end{align*}
Note that, $H_{n}(x,x,0)=0$, in addition, for any $x,z$ and $s\in[0,1]$
\[
\partial_{z}H_{n}(x,z,s)=- (f_*^n)'(z).
\]
Accordingly, by the Implicit Function Theorem any $n\in\bN$ and $s\in [0,1]$, there exists $Y_n(x,s)$ such
that\footnote{ The Implicit Function Theorem allows to define $Y_n(x,s)$ in a neighborhood of $s=0$; in fact we claim
  that this neighborhood necessarily contains $[0,1]$.  Otherwise, there would exist $\bar s\in(0,1)$ and $\bar x$ so
  that $Y_n$ is defined at $(\bar x,\bar s)$ but not at $(\bar x,s)$ with $s>\bar s$.  We then could apply the Implicit
  Function Theorem at the point $(\bar x,Y_n(\bar x,\bar s),\bar s)$ and obtain, by uniqueness, an extension of the
  previous function $Y_n$ to a larger neighborhood of $s=0$, which contradicts our assumption.}
$H_{n}(x,Y_n(x,s),s)=0$; from now on $Y_n(x)$ stands for $Y_n(x,1)$.  Note that setting $x^*_k=f_*^k(Y_n(x))$, by
construction, $x^*_n=x_n$. Observe moreover that
\begin{equation}\label{eq:Y-der}
\partial_x Y_n = (f_*^n)'(z)\invr \deh (\pi_x F_{\ve}^n) =
\frac{(1-G'_\ell\sigmaold_n)\nuold_n}{(f_*^n)'\circ Y_n},
\end{equation}
where we have used the notations introduced in equation~\eqref{e_defineSlopes}.  Recalling
\eqref{e_***} and by the cone condition we have
\begin{equation}\label{eq:deriv-Y0}
e^{-c_\# \ve n}\prod_{k=0}^{n-1} \frac{\partial_x f(x_k,\theta_k)}{f_*'(x^*_k)} \leq
\left| \frac{(1-G'_\ell\sigmaold_n)\nuold_n}{(f_*^n)'}\right| \leq e^{c_\# \ve n}\prod_{k=0}^{n-1}
\frac{\partial_x f(x_k,\theta_k)}{f_*'(x^*_k)}.
\end{equation}

Next, we want to estimate to which degree $x^*_k$ shadows the true
trajectory.
\begin{lem}\label{lem:shadow-q} There exists $C>0$ such that, for each $k\leq n< C \ve^{-\frac 12}$ we have
\[
\begin{split}
&\|\theta_k-\theta^*\|\leq \Const \ve k\\
&|x_k-x^*_k|\leq C_\# \ve k.
\end{split}
\]
\end{lem}
\begin{proof}
Observe that
\[
\theta_k=\ve\sum_{j=0}^{k-1}\omega(x_j,\theta_j)+\theta_0
\]
thus $\|\theta_k-\theta^*\|\leq \Const \ve k$.  Accordingly, let us set\footnote{ Here, as we already done before, we
  are using the fact that we can lift $\bT^1$ to the universal covering $\bR$.} $\xi_k=x_k^*-x_{k}$; then, by the mean
value theorem,
\begin{align*}
 | \xi_{k+1}| &= |\partial_x f\cdot \xi_k+\partial_\theta f\cdot (\theta_k-\theta^*)|\\
  &\geq \lambda |\xi_k|-\Const \ve k.
\end{align*}
Since, by definition, $\xi_n=0$, we can proceed by backward induction, which yields
\[
|\xi_k|\leq \sum_{j=k}^{n-1}\lambda^{-j+k}\Const \ve j\leq \Const \ve\sum_{j=0}^{\infty}\lambda^{-j}(j+k)\leq \Const \ve k.\qedhere
\]
\end{proof}
\begin{lem}\label{lem:der-bounds-geo}
There exists $C>0$ such that, for each $ n\leq C \ve^{-\frac 12}$,
\begin{equation}\label{eq:deriv-Y}
e^{-c_\# \ve n^2} \leq | Y'_n| \leq e^{c_\# \ve n^2}.
\end{equation}
In particular, $Y_n$ is invertible with uniformly bounded derivative.
\end{lem}
\begin{proof}
Let us prove the upper bound, the lower bound being similar. By equations \eqref{eq:Y-der}, \eqref{eq:deriv-Y0} and Lemma \ref{lem:shadow-q} we have
\[
|Y'_n|\leq e^{c_\# \ve n}e^{\sum_{k=0}^{n-1}
\ln\partial_x f(x_k,\theta_k)-\ln  f_*'(x^*_k)}\leq e^{\const \ve n}e^{\const\sum_{k=0}^{n-1}\ve k}.\qedhere
\]
\end{proof}
%%%%%%%%%%%%%%%%%%%%%%%%%%%%%%%
\section{Martingales, operators and It\=o's calculus}\label{sec:moi}
Suppose that $\cL_t\in L(\cC^r(\bR^d,\bR),\cC^0(\bR^d,\bR))$, $t\in\bR$, is a one parameter family of bounded linear
operators that depends continuously on $t$.\footnote{ Here $\cC^r$ are thought as Banach spaces, hence consist of
  bounded functions.  A more general setting can be discussed by introducing the concept of a {\em local martingale.}}
Also suppose that $\bP$ is a measure on $\cC^0([0,T],\bR^d)$ and let $\cF_t$ be the $\sigma$-algebra generated by the
variables $\{z(s)\}_{s\leq t}$.\footnote{ At this point the reader is supposed to be familiar with the intended meaning:
  for all $\vartheta\in \cC^0([0,T],\bR^d)$, $[z(s)](\vartheta)=z(\vartheta,s)=\vartheta(s)$.}

\begin{lem}\label{lem:basic-m} The two properties below are equivalent:
\begin{enumerate}
\item For all $A\in\cC^1(\bR^{d+1},\bR)$, such that, for all $t\in \bR$, $A(t,\cdot)\in \cC^r(\bR^{d},\bR)$, and for all times $s,t\in [0,T], s<t$, the function $g(t)=\bE(A(t, z(t))\;|\;\cF_s)$ is differentiable and
$g'(t)=\bE(\partial_t A(t,z(t))+\cL_t A(t,z(t))\;|\;\cF_s)$.
\item For all $A\in\cC^r(\bR^d,\bR)$, $M(t)=A(z(t))-A(z(0))-\int_0^t\cL_s A(z(s)) ds$ is a martingale with respect to $\cF_t$.
\end{enumerate}
\end{lem}
\begin{proof}
  Let us start with $(1)\Rightarrow (2)$.  Let us fix $t\in[0,T]$, then for each $s\in[0,t]$ let us define the random
  variables $B(s)$ by
\[
B(s,z)=A(z(t))-A(z(s))-\int_s^t\cL_\tau A(z(\tau))d\tau.
\]
Clearly, for each $z\in\cC^0$, $B(s,z)$ is continuous in $s$, and $B(t,z)=0$.  Hence, for all $\tau\in (s,t]$, by Fubini we have\footnote{ If uncomfortable about applying Fubini to conditional expectations, then have a look at \cite[Theorem 4.7]{V2}.}
\[
\begin{split}
\frac d{d\tau}\bE(B(\tau)\;|\; \cF_s)&=-\frac d{d\tau}\bE(A(z(\tau))\;|\; \cF_s)-\frac d{d\tau}\int_\tau^t\bE(\cL_r A(z(r))\;|\; \cF_s)dr\\
&=\bE(-\cL _\tau A(z(\tau))+\cL_\tau A(z(\tau))\;|\;\cF_s)=0.
\end{split}
\]
Thus, since $B$ is bounded, by Lebesgue dominated convergence theorem, we have
\[
0=\bE(B(t)\;|\; \cF_s)=\lim_{\tau\to 0}\bE(B(\tau)\;|\; \cF_s)=\bE(B(s)\;|\; \cF_s).
\]
This implies
\[
\bE(M(t)\;|\;\cF_s)=\bE(B(s)\;|\; \cF_s)+M(s)=M(s)
\]
as required.

Next, let us check $(2)\Rightarrow (1)$.  For each $h>0$ we have
\[
\begin{split}
\bE(A(t+h,z(t+h))-&A(t,z(t))\;|\;\cF_s)=\bE\left((\partial_t A)(t, z(t+h))\;|\;\cF_s\right)h+o(h)\\
&+\bE\left(M(t+h)-M(t)+\int_t^{t+h}\cL_\tau A(t, z(\tau))d\tau\;|\;\cF_s\right).
\end{split}
\]
Since $M$ is a martingale $\bE\left(M(t+h)-M(t)\;|\;\cF_s\right)=0$.  The lemma follows by Lebesgue dominated
convergence theorem.
\end{proof}

The above is rather general, to say more it is necessary to specify other properties of the family of operators $\cL_s$. A case of particular interest arises for second order differential operators like \eqref{eq:secondorderop}.
Namely, suppose that
\[
(\cL_s A)(z)= \sum_{i} a(z,s)_i\partial_{z_i}A(z)+\frac 12\sum_{i,j=1}^d[\sigma^2(z, s))]_{i,j} \partial_{z_i}\partial_{z_j}A(z),
\]
where, for simplicity, we assume $a, \sigma$ to be smooth and bounded and $\sigma_{ij}=\sigma_{ji}$. Clearly, \eqref{eq:secondorderop} is a special case of the above.
In such a case it turns out that it can be established a strict connection between $\cL_s$ and the Stochastic Differential Equation
\begin{equation}\label{eq:sde-general}
 d z= a dt+ \sigma dB
\end{equation}
where $B$ is the standard Brownian motion. The solution of \eqref{eq:sde-general} can be defined in various way. One possibility is to define it as the solution of the Martingale problem \cite{SV}, another is to use stochastic integrals \cite[Theorem 6.1]{V2}. The latter, more traditional, approach leads to It\=o's formula that reads, for each bounded continuous function $A$ of $t$ and $z$, \cite[page 91]{V2},
\[
\begin{split}
A(z(t),t)-A(z(0),0)=&\int_0^t \partial_s A(z(s),s)ds+\int_0^t \sum_i a(z(s),s)_i\partial_{z_i} A(z(s),s) ds\\
&+\frac 12\int_0^t\sum_{i,j}\sigma^2(z(s),s)\partial_{z_i}\partial_{z_j} A(z(s),s) ds\\
&+\sum_{i,j}\int_0^t\sigma_{ij}(z(s),s)\partial_{z_j}A(z(s),s)_j dB_i(s)
\end{split}
\]
where the last is a stochastic integral \cite[Theorem 5.3]{V2}. This formula is often written in the more impressionistic form
\[
dA=\partial_t A dt+a\partial_z A dt+\sigma\partial_z A dB+\frac 12\sigma^2\partial_z^2 A dt = \partial_t A dt+\sigma\partial_z A dB+\cL_t A dt.
\]
Taking the expectation with respect to $\bE(\cdot\;|\;\cF_s)$ we obtain exactly condition (1) of Lemma \ref{lem:basic-m}, hence we have that the solution satisfies the Martingale problem.

\begin{rem} Note that, if one defines the solution of \eqref{eq:sde-general} as the solution of the associated Martingale problem, then one can dispense from It\=o's calculus altogether. This is an important observation in our present context in which the fluctuations come form a deterministic problem rather than from a Brownian motion and hence a direct application of It\=o's formula is not possible.
\end{rem}

\end{document}